\pdfoutput=1
\documentclass[11pt,a4paper]{article}

\usepackage{amsfonts,amsmath,amssymb,amstext,amsthm}  
\usepackage{lmodern,microtype}
\usepackage{delimdelim}
\usepackage[english]{babel}
\usepackage[T1]{fontenc}
\usepackage{graphicx,wrapfig}
\usepackage[font=sf]{caption}
\graphicspath{{./figs/}} 
\usepackage{fullpage,xspace,csquotes}
\newtheorem{prop}{Proposition}[section] 

\newtheorem{lem}[prop]{Lemma}
\newtheorem{theo}[prop]{Theorem}
\newtheorem{cor}[prop]{Corollary}
\newtheorem{defin}[prop]{Definition}

\numberwithin{equation}{section}
\newcommand{\rhs}{f}
\newcommand{\bc}{g}
\usepackage[natbib, %
url=false,hyperref=true,%
eprint=true,%
clearlang=true,%
autocite=inline,
giveninits=true,%
maxnames=4,date=year,%
doi=true,%
isbn=false,%
sortcites,%
citestyle=authoryear-comp,style=authoryear
]{biblatex} %
\usepackage{hyperref}
\newcommand{\beq}{\begin{eqnarray}}
\newcommand{\beqq}{\begin{eqnarray*}}
\newcommand{\eeq}{\end{eqnarray}}
\newcommand{\eeqq}{\end{eqnarray*}} 

\newcommand{\iid}{\emph{iid}\xspace}

\newcommand{\lln}{strong law of large numbers\xspace}
\newcommand{\clt}{central limit theorem\xspace}

\title{Unbiased `walk-on-spheres' Monte Carlo methods\\ for the fractional Laplacian}
\author{ {Andreas E. Kyprianou}\thanks{Department of Mathematical Sciences, University of Bath, Claverton Down, Bath, BA2 7AY, UK. } $^,$\thanks{Supported by EPSRC grant EP/L002442/1.} \and
   Ana Osojnik\thanks{Mathematical Institute,
University of Oxford,
Woodstock Road,
Oxford
OX2 6GG, UK.\newline
\mbox{\quad} Email:
\texttt{a.kyprianou@bath.ac.uk}, \texttt{anaosojnik@gmail.com}, 
\texttt{t.shardlow@bath.ac.uk}. }
    \and
   {Tony Shardlow$^*$}
   }
 \date{\today}

\begin{document}

\maketitle
\begin{abstract}
	\noindent We consider Monte Carlo methods for simulating solutions to the analogue of the Dirichlet boundary-value problem in which the Laplacian is replaced by the fractional Laplacian and boundary conditions are replaced by conditions on the exterior of the domain. Specifically, we consider the analogue of the  so-called  `walk-on-spheres' algorithm. In the diffusive setting, this entails sampling the path of Brownian motion as it uniformly exits a sequence of spheres maximally  inscribed in the domain. As this algorithm would otherwise never end, it is truncated when the `walk-on-spheres' comes within $\varepsilon>0$ of the boundary.  In the setting of the fractional Laplacian, the role of Brownian motion is replaced by an isotropic $\alpha$-stable process with $\alpha\in(0,2)$. A significant difference to the Brownian setting is that the stable processes will exit spheres by a jump rather than hitting their boundary. This difference ensures that disconnected domains may be considered and that, unlike the diffusive setting, the algorithm ends after an almost surely finite number of steps. 
\end{abstract}

\section{Introduction}

We start by  recalling the classical Dirichlet problem in $d$-dimensions and re-examining a, now, classical Monte Carlo algorithm that is used to numerically simulate its solution. Suppose that $D$ is a  domain in $\mathbb{R}^d$, $d\geq 2$,  with sufficiently smooth boundary.  We are interested in finding $u\colon D\to \mathbb{R}$ such that
\begin{gather}
	\begin{aligned}
		\Delta u(x) & = 0,    & \qquad  x & \in D,          \\
		u(x)        & = \bc(x), & x         & \in \partial D, 
	\end{aligned}\label{Dirichlet}
\end{gather} 
where $\bc$  is a  given continuous function on the boundary. 
Feynman--Kac representation tells us that, for example, if $u\in C^2(\overline{D} )$ is a solution to \eqref{Dirichlet}, then
\begin{equation}
	\label{FK}  
	u(x) = \mathbb{E}_x[\bc(W_{\tau_D}) ],
	\qquad x\in D,
\end{equation} 
where $\tau_D \coloneqq \inf\{t>0 : W_t \not\in D\}$ and $W\coloneqq (W_t, t\geq 0)$ is  standard $d$-dimensional Brownian motion with probabilities $(\mathbb{P}_x, x\in\mathbb{R}^d)$. 
 
The representation \eqref{FK} suggests that solutions to \eqref{Dirichlet} can be generated numerically via straightforward Monte Carlo simulations of the path of $W$ until first exit from $D$. That is to say, if $(W^{i}_t, t\leq \tau^{i}_D)$, $i \in \mathbb{N}$ are \iid copies of $(W_t, t\leq \tau_D)$ issued from $x\in D$, then, by the  \lln, 
\begin{equation}
	\lim_{n\to\infty}\frac{1}{n}\sum_{i=1}^n \bc(W^{i}_{\tau^{i}_D}) 
	= u(x),\qquad \text{almost surely.}
	\label{FKMC}
\end{equation}
For practical purposes, since it is impossible to take the limit, one truncates the series of estimates for large $n$ and the \clt gives $\mathcal{O}(1/n)$ upper bounds on the variance of the $n$-term sum, which serves as a numerical error estimate.

Although forming the fundamental basis of most Monte Carlo methods for diffusive Dirichlet-type problems, \eqref{FKMC} is an inefficient numerical approach. Least of all, this is because the Monte Carlo simulation of $u(x)$ is independent for each $x\in D$. Moreover, it is unclear how  exactly to simulate the path of a Brownian motion on its first exit from $D$, that is to say, the quantity $W_{\tau_D}$. This is because of  the fractal properties of Brownian motion, making its path difficult to simulate. This introduces additional numerical errors over and above that of Monte Carlo simulation.

A method proposed by \autocite{Mu}, for the case that $D$ is convex,  sub-samples special points along  the path of Brownian motion to the boundary of the domain $D$. The method does not require a complete simulation of its path and  takes advantage of the distributional symmetry of Brownian motion. In order to describe the so-called `walk-on-spheres', we need to first introduce some notation. We may thus  
set $\rho_0 = x$ for $x\in D$ and define $r_1$  to be the radius of the largest sphere inscribed in $D$ that is centred at $x$. This sphere we will call $S_1 = \{y\in\mathbb{R}^d\colon |y-\rho_0|  = r_1\}$. To avoid special cases, we henceforth assume that the surface area of $S_1\cap \partial D$ is zero (this excludes, for example, the case that $x = 0$ and $D$ is a sphere centred at the origin).

Now set $\rho_1\in D$ to be a point uniformly distributed on $S_1$ and note that, given the assumption in the previous sentence, $\mathbb{P}_x(\rho_1\in \partial D) = 0$. Construct the remainder of the sequence $(\rho_n,\,n
\geq 1)$ inductively. Given $\rho_{n-1}$, we  define the radius, $r_n$, of the largest sphere inscribed in $D$ that is centred at $\rho_{n-1}$. Calling this sphere $S_n$, we have that $S_n = \{y\in \mathbb{R}^d \colon |y-\rho_{n-1}| =r_n\}$. We now select $\rho_n$ to be a point that is uniformly positioned on $S_n$. Once again, we note that if $\rho_{n-1}\in D$ almost surely, then the uniform distribution of both $\rho_{n-1}$ and $\rho_{n}$  ensures that $\mathbb{P}(\rho_{n}\in \partial D) = 0$.
Consequently, the sequence $\rho_n$ continues for all $n\geq1$. In the case that $\rho_n$ approaches the boundary, the sequence of spheres $S_n$  become arbitrarily small in size.

Thanks to the strong Markov property and the stationary and independent increments of Brownian motion, it is straightforward to prove the following result.
 
\begin{lem}
	Fix $x\in D$ and define $\rho'_1 = W_{\tau_{S'_1}}$, where $\tau_{S'_1}=\inf\{t>0 \colon
	 W_t \in S'_1 \}$ and $S'_1$ is the largest sphere, centred at $x$, inscribed in $D$. For $n\geq 2$, given $\rho'_{n-1}\in D$, let $\rho'_n = W_{\tau_{S'_n}}$, where $\tau_{S'_n}=\inf\{t>0 \colon W_t \in S'_n \}$ and $S'_n$ is the largest sphere, centred at $\rho'_{n-1}$. Then the sequences $(\rho_n, n\geq 0)$ and $(\rho'_n, n\geq 0)$ have the same law. 
\end{lem}

\noindent As an immediate consequence, $\lim_{n\to\infty}\rho_n$ almost surely exists and, moreover, it it equal in distribution to $W_{\tau_D}$. The sequence $\rho\coloneqq (\rho_n, n\geq 0)$ may now replace the role of $(W_t, t\leq \tau_D)$ in \eqref{FK}, and hence in (\ref{FKMC}), albeit that one must stop the sequence $\rho$ at some finite $N$. By picking a threshold $\varepsilon>0$, we can choose $N(\varepsilon)$ as a cutoff for the sequence $\rho$ such that $N(\varepsilon) = \min\{n\geq 0: \inf_{z\in\partial D}|\rho_n -z|\leq \varepsilon\}$. Intuitively, one is inwardly `thickening' the boundary $\partial D$ with an `$\varepsilon$-skin' and stopping once the walk-on-spheres hits the  $\varepsilon$-skin. As the sequence $\rho$ is random, $N(\varepsilon)$ is also random.  Starting with Theorem 6.6 of \autocite{Mu} and the classical computations in \autocite{Mo}, it is known that  
$
\mathbb{E}_x[N(\varepsilon)] = \mathcal{O}(|\!\log\varepsilon|).
$
To be more precise, we have the following result.

\begin{theo}\label{muller} Suppose that $D$ is a convex domain.
	There exist constants $c_1,c_2>0$ such that $\mathbb{E}_x[N(\varepsilon)] \leq c_1\,\abs{\log \varepsilon} + c_2$, $\varepsilon\in(0,1)$. 
\end{theo}

\noindent The Monte Carlo simulation (\ref{FKMC}) can now be replaced by one based on simulating the quantity 
$
\bc(\rho_{N(\varepsilon)})$, $\rho_0 = x\in D$,
which, in turn, is justified  by the \lln:
\begin{equation}
	\label{WoSMC1}
	\lim_{n \to\infty} \frac{1}{n}\sum_{i = 1}^n \bc(\rho^{i}_{N^{i}(\varepsilon)}) %
	= \mathbb{E}_x \bp{\bc(\rho_{N(\varepsilon)})}%
	\approx \mathbb{E}_x[\bc(W_{\tau_D})]= u(x),\qquad{\text{a.s.,}}
\end{equation}
where $\varepsilon>0$ is some threshold and $(\rho^{i}_n, n\leq N^{i}(\varepsilon))$, $i\geq 0$ are \iid copies of the walk-on-spheres process stopped at a distance $\varepsilon$ or smaller from $\partial D$. Formally speaking,  a convention is required to evaluate $\bc$ just inside the boundary $\partial D$ in (\ref{WoSMC1}). In many cases, $\bc$ can be evaluated without introducing any additional bias \autocite{Given1997-qr,Hwang2001-wc}.
		
\bigskip
		
The Laplacian serves as the infinitesimal generator of Brownian motion, in the sense that, for appropriately smooth functions $\phi\colon \mathbb{R}^d\to \mathbb{R}$,
\begin{equation}
	\lim_{t\to0} \frac{\mathbb{E}_x[\phi(W_t)] - \phi(x)}{t} =\frac{1}{2}\Delta\phi(x), \qquad x\in\mathbb{R}^d.
	\label{IG}
\end{equation}
Intuitively speaking, this explains an underlying connection between the Dirichlet problem (\ref{Dirichlet}) and the Feynman--Kac representation of the solution (\ref{FK}).
In this paper, we consider the analogue of (\ref{Dirichlet})  when the operator $ \Delta/2$ is replaced by the fractional Laplacian $-(-\Delta)^{\alpha/2}$ for $\alpha\in (0,2)$. In this case, the fractional Laplacian corresponds, in the same sense as (\ref{IG}), to an isotropic stable L\'evy process with index $\alpha$. This is a strong Markov process  with stationary and independent increments, say $X = (X_t, t\geq 0)$ with probabilities $(\mathbb{P}_x, x\in\mathbb{R}^d)$, whose semi-group is represented by the Fourier transform
		
\[
	\mathbb{E}_0\bp{{\rm e}^{{\rm i}\langle\theta ,X_t\rangle}}%
	= {\rm e}^{-|\theta|^\alpha t}, \qquad \theta\in \mathbb{R}^d, \;t\geq 0,
\]
where $\langle \cdot,\cdot \rangle$ represents the usual Euclidian inner product. Stable processes enjoy an isotropy in the following sense: if $U$ is any orthogonal matrix in $\mathbb{R}^{d\times d}$, then $(UX_t, t\geq 0)$ under $\mathbb{P}_0$ has the same law as $(X, \mathbb{P}_0)$. Moreover,  we have the following important  scaling property: for all $c>0$, 
\begin{equation}\label{levyscaling}
	((cX_{c^{-\alpha} t}, t\geq 0), \mathbb{P}_0) \text{ is equal in law to }((X_t, t\geq 0), \mathbb{P}_0).
\end{equation}
		
In dimension two or greater,  the operator $-(-\Delta)^{\alpha/2}$ can be expressed in the form  
\[
	-(-\Delta)^{\alpha/2} u(x)  =-\frac{2^\alpha \,\Gamma((d+\alpha)/2)}{\pi^{d/2}\,\Gamma(-\alpha/2)} \lim_{\varepsilon\downarrow0}\int_{\mathbb{R}^d\backslash B(0,\varepsilon)}\frac{[u(y)- u(x)]}{|y-x|^{d+\alpha}}\,{\rm d}y,\qquad x\in \mathbb{R}^d,
\]
where $B(0,\varepsilon) = \{x\in\mathbb{R}^d: |x|<\varepsilon\}$ and $u$ is smooth enough for the limit to make sense.

Noting that $-(-\Delta)^{\alpha/2}$ is no longer a local operator, the analogous formulation of (\ref{Dirichlet}) needs a little more care. In particular, the boundary condition on the domain $D$ is no longer stated on $\partial D$, but must now be stated on the complement of $D$, written $D^{\rm c}$. To avoid pathological cases, we must assume throughout that $D^{\rm c}$ has positive $d$-dimensional Lebesgue measure.  The Dirichlet problem for $-(-\Delta)^{\alpha/2}$ requires one to find $u\colon D \to\mathbb{R}$ such that
\begin{gather}
	\begin{aligned}
		-(-\Delta )^{\alpha/2}u(x) & = 0, & \qquad x & \in D, &   \\
		u(x)                       & = \bc(x),      & x      & \in  D^{\rm c}, 
	\end{aligned} 
	\label{aDirichlet}
\end{gather}
where $\bc$ is a  suitably regular function. 
The fractional Dirichlet problem and variants thereof appear in many applications, in particular in physical settings where 
anomalous dynamics occur and where  the spread of mass grows faster than linearly in time. Examples include  turbulent fluids, contaminant transport in fractured rocks, chaotic dynamics and disordered quantum ensembles; see \citep{FracDy, AnTr, flights}. The numerical analysis of \eqref{aDirichlet} is no less deserving than in the diffusive setting.

Just as with the classical Dirichlet setting, the solution to \eqref{aDirichlet} has a Feynman--Kac representation, expressed as an expectation at first exit from $D$ of the associated stable process.
The  theorem below is proved in this paper in a probabilistic way. Similar statements and  proofs we found in the existing literature  take a more analytical perspective.  See for example the review in \autocite{bucur} as well as the monographs \autocite{BH}, \autocite{BV} and \autocite{BBKRSV}, the articles \autocite{B99}, \autocite{R-O1} and \autocite{R-O2} and references therein.
		
We say a real-valued function $\phi$ on a Borel set $S\subset \mathbb{R}^d$ belongs to $L^1_\alpha(S)$ if it is a measurable function that satisfies  
\begin{equation}
	\int_{S}\frac{|\phi(x)|}{1+|x|^{\alpha + d}}\,{\rm d}x <\infty.
	\label{i-test}
\end{equation}
		
\begin{theo}\label{corr} For dimension $d\geq 2$, suppose that $D$ is a  bounded domain in $\mathbb{R}^d$ and that $\bc$ is a continuous function in $L^1_\alpha(D^\mathrm{c})$.
	Then there exists a  unique continuous solution to   \eqref{aDirichlet}  in $L^1_\alpha(\mathbb{R}^d)$, which is given by 
																																																																														            
	\[
		u(x)  = \mathbb{E}_x[\bc(X_{\sigma_D})],\qquad x\in D,
	\]
	where $X = (X_t, t\geq 0)$ is an isotropic stable L\'evy process with index $\alpha$ and $\sigma_D = \inf\{t>0: X_t\not\in D\}$. 
\end{theo}
\noindent The case that $D$ is a ball can be found, for example, in Theorem 2.10 of \autocite{bucur}. We exclude the case $d=1$ because convex domains are intervals for which exact solutions are known; see again \autocite{bucur} or the forthcoming Theorem \ref{BGR} lifted from \autocite{BGR}.		
Theorem \ref{corr} follows in fact as a corollary of a more general result stated later in Theorem \ref{hasacorr}, which is proved in the Appendix.
		
		
\bigskip
		
In this article, our objective is to demonstrate that the walk-on-spheres method may also be extended to the setting of the Dirichlet problem with fractional Laplacian. In particular, we will show that, thanks to various distributional and path properties of stable processes, notably spatial homogeneity, isotropy, self-similarity and that it exits $D$ by a jump, simulations can be made unbiased, without the need to truncate the algorithm at an $\varepsilon$ tolerance. 
Whilst there exist many methods for numerically examining the fractional Dirichlet problem \eqref{aDirichlet}, which mostly appeal to classical methodology for diffusive operators, see for example \autocite{NOS, HO, DEG, ZRK, BS, SB,1608.08443} to name some but not all of the existing literature, we believe that no other work appeals to the walk-on-spheres algorithm in this context.

The remainder of this paper is structured as follows. 
In the next section, we give a brief historical review of Theorem \ref{muller} and its proofs as well as providing a new, short proof.
In Section~\ref{stable_paths}, we show how an old result of \autocite{BGR} can be used to give an exact simulation of the paths of stable processes. 
In Section~\ref{WoSfL}, we introduce  the walk-on-spheres algorithm for the fractional-Laplacian Dirichlet problem. We start with domains $D$ that are convex but not necessarily bounded. Our main result shows that the walk-on-spheres algorithm ends in an almost-surely finite number of steps (without the need of approximation), which can be stochastically bounded by a geometric distribution. Moreover, the parameter of this  distribution does not depend on the starting point of the walk-on-spheres algorithm. Section~\ref{non_convex} looks at extensions to non-convex domains. In Section \ref{inhomogenous}, we consider a fractional Poisson equation, where  an inhomogeneous term is introduced on the right-hand side of the fractional-Laplacian Dirichlet problem (\ref{aDirichlet}). Appealing to related results concerning the resolvent of stable processes until first exit from the unit ball, we are able to develop the walk-on-spheres algorithm further. Finally in Section \ref{numerics}, we discuss some numerical experiments to illustrate the methods developed as well as their implementation.

\section{The classical setting}
		
As promised above, we give a brief historical review of the classical walk-on-spheres algorithm and, below, for completeness, we provide a proof of Theorem \ref{muller}, which, to the authors' knowledge, is new. 
The walk-on-spheres algorithm was first derived by  \autocite{Mu}. In Theorem 6.1 of his article, Muller claims that one can compare $\mathbb{E}_x[N(\varepsilon)]$ with the mean number of steps of a walk-on-spheres process that is stopped when it reaches an $\varepsilon$-skin of  the tangent hyperplane that passes through a point on $\partial D$ that is closest to $x$. Although the claim is correct (indeed the proof that we give for our main result Theorem \ref{main} below provides the basis for an alternative justification of this fact), it is not entirely clear from Muller's reasoning. \autocite{Mo} uses Muller's comparison of the mean number of steps to prove Theorem \ref{muller}. He considers the total expected occupation of an appropriately time-changed version of Brownian motion when crossing each sphere of the walk until touching the aforementioned $\varepsilon$-skin of  the tangent hyperplane. Using the  self-similarity of Brownian motion, Motoo argues that the time-change during passage  to the boundary of each  sphere is such that the expected occupation across each step is uniformly bounded below. It follows that the sum of these weighted expected occupations can be bounded below by $\mathbb{E}_x[N(\varepsilon)]$. On the other hand, the aforesaid sum can also be bounded above by the total expected time-changed occupation until exiting the half-space (as defined by the tangent plane), which can be computed explicitly, thereby providing the $|\!\log\varepsilon|$ comparison. 
		
Following the foundational work of Muller and Motoo, there have been many reproofs and generalisations of the original algorithm to different processes and domain types. 
Notable in this respect is the work of \autocite{Mi} and \autocite{BB} who consider non-convex domains and \autocite{Sa}, who appeals to renewal theory to analyse the growth in $\varepsilon$ of the mean number of steps to completion of the  walk-on-spheres algorithm. His method also allows for  variants of the algorithm in which the sphere sizes do not need to be optimally  inscribed in $D$. Later,  \autocite{ST} gives an elementary proof of the $|\!\log \varepsilon|$ bound.  Mascagni and co-authors have extensively developed the walk-on-spheres algorithm in applications; see for example \autocite{HMG,Given2001-lu,Given2002-vs,Mackoy2013-es,Hwang2001-wp}.

\begin{proof}[Proof Theorem \ref{muller}] We break the proof into two parts. In the first part, we analyse the walk-on-spheres process over one step, by considering the distance of the next point in the algorithm from the orthogonal tangent hyperplane of the first point. (Note the existence of a tangent hyperplane requires convexity of the domain.) In the second part of the proof, we use this analysis to build a supermartingale, from which the desired result follows via optional stopping.

	\medskip
																																																																														            
	For the first part of the proof, we start by introducing notation. For any $x = (x_1, \dots, x_d)\in\mathbb{R}^d$ such that $x_1>0$, let us write $V(x) = \{(z_1, \dots, z_d)\in\mathbb{R}^d\colon z_1>0\}$ for the open half-space containing $x$ and denote its boundary $\partial V(x) = \{(z_1, \dots, z_d)\in\mathbb{R}^d\colon z_1=0\}$. Suppose that we choose our coordinate system so that  $x\in D$ is such that $\rho_0 = x = (x_1, 0,\dots, 0)$ and $\partial V(\rho_0)$ is a tangent hyperplane to both $D$ and $S_1$. This assumption comes at no cost as, thanks to isotropy and spatial homogeneity of Brownian motion. Let us define $\zeta_0$, the  orthogonal distance of $\rho_0$ from $\partial V(\rho_0)$.  With the assumed choice of coordinate system,  write $\zeta_0 \coloneqq r_1 =  x_1 = |x| = |\rho_0|$ and define 
	\[
		\zeta_1 = \min\Bigl\{\varepsilon, \inf_{z\in\partial V(\rho_0)}|\rho_1-z|\Bigr\};
	\] that is, the minimum of $\varepsilon$ and the orthogonal distance of $\rho_1$ from $\partial V(\rho_0)$. Next, define $\theta_1$, the angle that subtends at $\rho_0$ between $\rho_1$ and the origin $(0,\dots,0)$ and recall that symmetry implies that $\theta_1$ is uniformly distributed on $[0,2\pi]$. Simple geometric considerations tell us that 
	\begin{equation}
		\zeta_1=x_1 - r_1 \sin\left(\frac{\pi}{2} - \theta_1\right) = \zeta_0 - \zeta_0 \sin\left(\frac{\pi}{2} - \theta_1\right) = \zeta_0(1-\cos(\theta_1)).
		\label{zeta1}
	\end{equation}
	This provides an implicit expression for $\theta_1$ in terms of the orthogonal distance $\rho_0$ from the nearest tangent hyperplane. 
	See Figure \ref{fig:class_proofa}.
	\begin{figure}[h]
		\centering
		\includegraphics[width=0.8\linewidth]{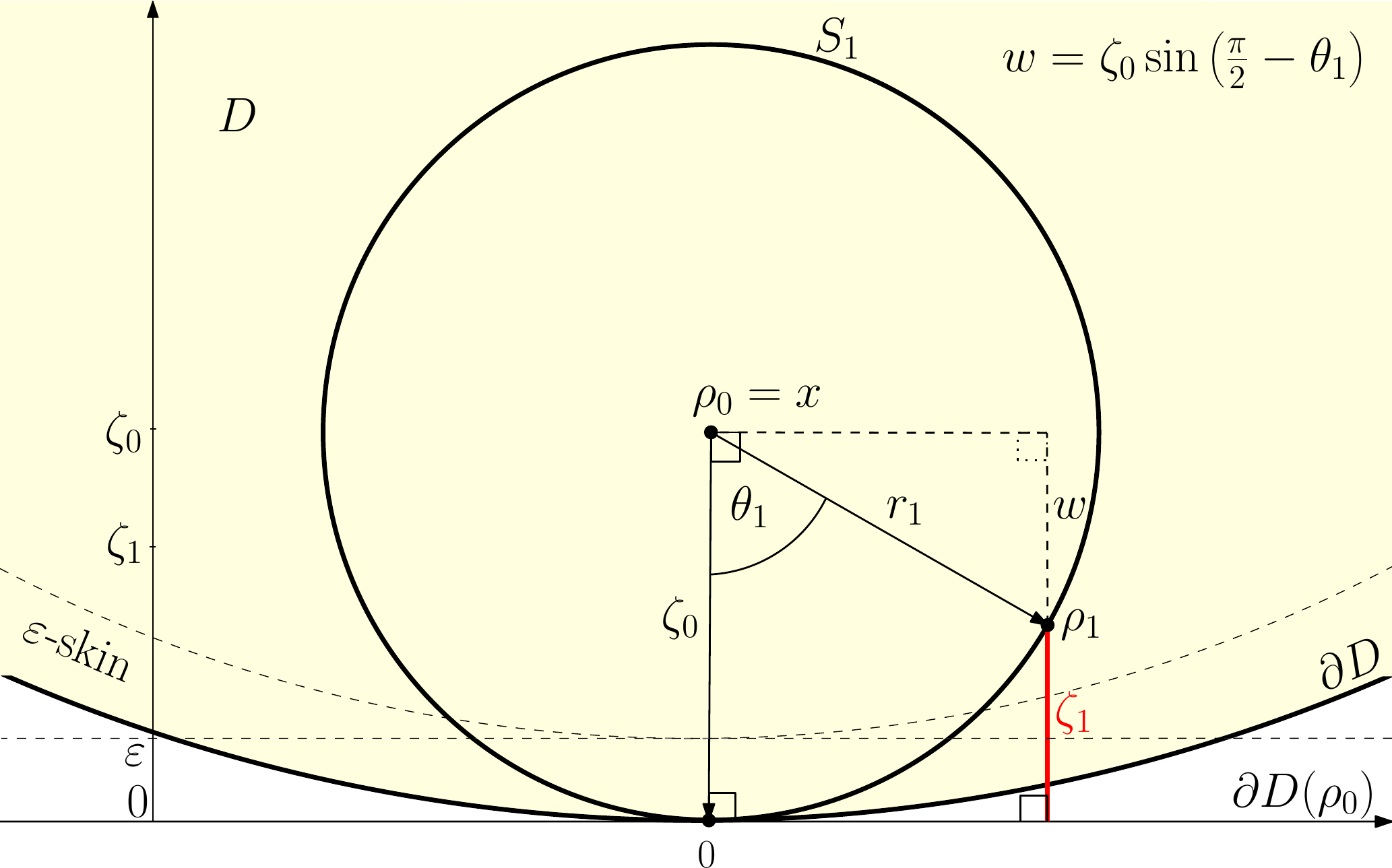}
		\caption{Geometric setting of the proof}
		\label{fig:class_proofa}
	\end{figure}
																																																																														            
	Assuming that $\zeta_0>\varepsilon$, thanks to isotropic symmetry, the walk-on-sphere algorithm will end at the first step if $\theta_1$ lies in a certain critical interval dictated by the choice of skin thickness $\varepsilon$. We can compute this critical (and obviously) symmetric interval as a function of $\zeta_0$, say $(-\theta^*(\zeta_0), \theta^*(\zeta_0))$, where 
	\begin{equation}
		\theta^*(\zeta_0) = \arccos\pp{\frac{\zeta_0 - \varepsilon}{\zeta_0}}.
		\label{theta*}
	\end{equation}
																																																																														            
	A quantity that will be of interest to us in order to complete the proof is the expectation 
	$
	\mathbb{E}_{x}[\sqrt{\zeta_1}] = \mathbb{E}_{\rho_0}[\sqrt{\zeta_1}].
	$
	To this end, we compute
	\begin{align}
		\mathbb{E}_{\rho_0}\left[\sqrt{\zeta_1}\right] & \leq \sqrt{\varepsilon}\,\mathbb{P}_{\rho_0}\pp{\theta_1\in (-\theta^*(\zeta_0), \theta^*(\zeta_0))}+ 
		\mathbb{E}_{\rho_0}\left[ \mathbf{1}_{(\theta_1\not\in (-\theta^*(\zeta_0), \theta^*(\zeta_0)))}\sqrt{\zeta_1}\right]\notag\\
		                                               & = \sqrt{\varepsilon}\, \frac{\theta^*(\zeta_0)}{\pi}                                               
		+ \frac{1}{\pi}\int_{\theta^*(\zeta_0)}^\pi \sqrt{\zeta_0(1-\cos(u))}\,{\rm d}u\notag\\
		                                               & \eqqcolon\Lambda({\varepsilon}/{\zeta_0})\,\sqrt{\zeta_0},                                         
		\label{ineq}
	\end{align}
	where $\mathbf{1}_S$ denotes the indicator function on the set $S$. Using the primitive ${\displaystyle\int} \sqrt{1- \cos(u)}\,{\rm d}u$ $ =-2\, \sqrt{1-\cos(u)}\,\cot(u/2)$, we have
	\[
		\Lambda(u)  =\sqrt{u}\,\frac{\arccos(1-u)}{\pi}
		+\frac{2}{\pi}\sqrt{u} \,\cot\pp{\frac{\arccos(1-u)}{2}}.
	\]  
	One easily verifies that there is a constant $\lambda\in(0,1)$ such that $\sup_{u\in [0,1]}\Lambda(u)<\lambda$.
																																																																														            
	\medskip
																																																																														            
	Next we move to the second part of the proof. At each step of the walk-on-spheres,  we can construct the quantities $\zeta_{n+1}$, the orthogonal distance of  $\rho_{n+1}$ to the tangential hyperplane that passes through the closest point on $\partial D$ to $\rho_n$; and $\theta_n$, the angle that is subtended at $\rho_n$ between the aforesaid point and $\rho_{n+1}$. Note that $\varepsilon$ is an absorbing state for the sequence $(\zeta_n, n\geq 0)$ in the sense that, if $\zeta_n = \varepsilon$, then $\zeta_{n+k} = \varepsilon$ for all $k\geq 0$. We may thus write $N(\varepsilon) \leq  N'(\varepsilon):=\min\{n\geq 0: \zeta_n = \varepsilon\}$.
																																																																														            
	By the strong Markov property and the spatial homogeneity of Brownian motion given the analysis leading to (\ref{ineq}), we have, on $\{n<N(\varepsilon)\}$,
	\[
		\mathbb{E}\left[\left.\sqrt{\zeta_{(n+1)\wedge N(\varepsilon)}}\,\right|\zeta_0, \dots, \zeta_n\right] =  \mathbb{E}\left[\left.\sqrt{\zeta_{(n+1)\wedge N(\varepsilon)}}\,\right| \zeta_n\right] \leq \Lambda({\varepsilon}/{\zeta_n})\sqrt{\zeta_n}< \lambda \sqrt{\zeta_n}.
	\]
	As a consequence the process $\left(\lambda^{-(n\wedge N(\varepsilon))}\sqrt{\zeta_{n\wedge N(\varepsilon)}}, n\geq 0\right)$ is a supermartingale. The optional-sampling theorem and Jensen's inequality give us
	\[
		\varepsilon \lambda^{-\mathbb{E}_x[N'(\varepsilon)]}\geq \mathbb{E}_{x}[\lambda^{-N'(\varepsilon)}\varepsilon]\leq \sqrt{r_1},
		\qquad x\in D.
	\]
	The result now follows by taking logarithms.
\end{proof}

\section{Exact simulation of stable paths}\label{stable_paths}
		
The key ingredient to the walk-on-spheres in the Brownian setting is the knowledge that spheres are exited continuously and uniformly on the boundary of spheres. In the stable setting, the inclusion of path discontinuities means that the process will exit a sphere by a jump. 
The analogous  key observation that makes our analysis possible is the following result, which gives the distribution of a stable process issued from the origin, when it first exits a unit sphere.
		
\begin{theo}[Blumenthal, Getoor, Ray, 1961]\label{BGR}
	Suppose that $B(0,1)$ is a unit ball centred at the origin and write $\sigma_{B(0,1)} = \inf\{t>0 : X_t \not\in B(0,1)\}$. Then,
	\[
		\mathbb{P}_0(X_{\sigma_{B(0,1)}}\in \mathrm{d}y) = \pi^{-(d/2+1)}\,\Gamma(d/2)\,\sin(\pi\alpha/2)\,\left|1-|y|^2\right|^{-\alpha/2}|y|^{-d}\,{\rm d}y, \qquad  |y|>1.
	\]
\end{theo}

This result provides a method of constructing precise sample paths of stable processes in phase space (i.e.\ exploring sample paths as ordered subsets of $\mathbb{R}^d$ rather than as functions $[0,\infty]\to \mathbb{R}^d$). Choose a tolerance $\epsilon$ and initial point $X_0 =x$. Denote by $E_1$ a  sampling from the distribution given in Theorem \ref{BGR}. This gives the exit from a ball of radius one when $X$ is issued from the origin. By the scaling property \eqref{levyscaling} and the stationary and independent increments, $x +\epsilon\, E_1$ is distributed as the exit position from a ball of radius $\epsilon$ centred at $x$ when the process is issued from $x$. Hence, we define $X_1=x + \epsilon\, E_1$ and then, inductively for $n\geq 1$,  generate $X_{n+1}$ as the exit point of the ball centred on $X_n$ with radius $\epsilon$ by noting this is equal in distribution to $X_n + \epsilon \, E_{n+1}$, where $E_{n+1}$ is an \iid copy of $E_1$. 
It is important to remark for later that the value of $\epsilon$ in this algorithm does not need to be fixed and may vary with each step. Note, however, the method does not generate the corresponding time to exit from each ball. Therefore, the sample paths that are produced, whilst being exact in the distribution of points that the stable process will pass through, cannot be
represented graphically in time as there is only an equal mean duration to exiting each sphere. If the tolerance $\epsilon$ is altered on each step, then  even this mean duration feature is lost. 
The method is used to generate  Figure~\ref{fig:sample_path}. 
		
\begin{figure}[h!]
	\centering
	\includegraphics{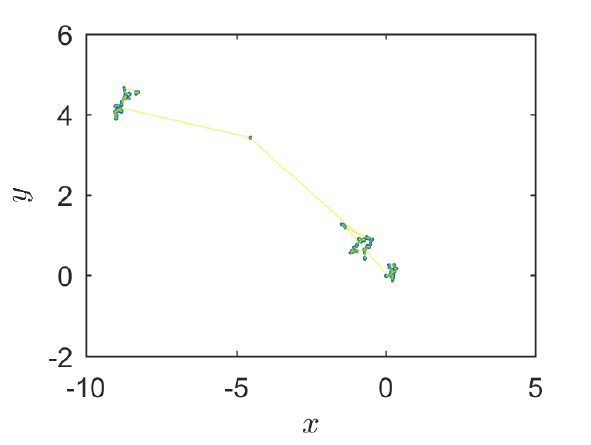}
	\includegraphics{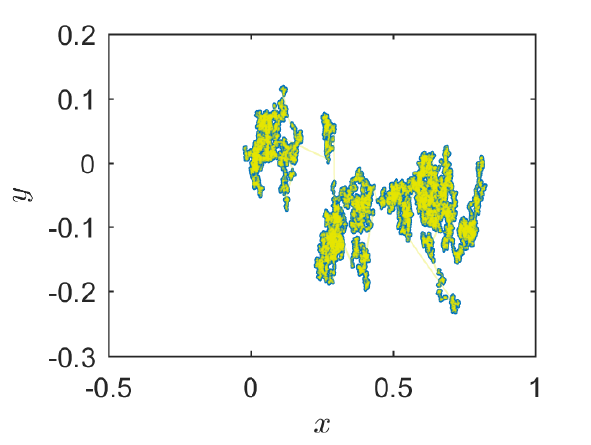}\\
	\includegraphics{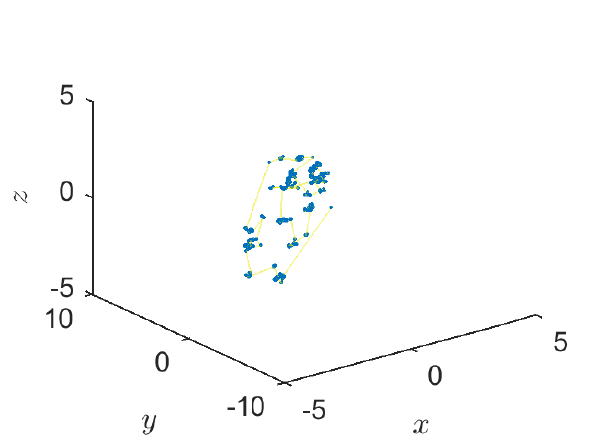}
	\includegraphics{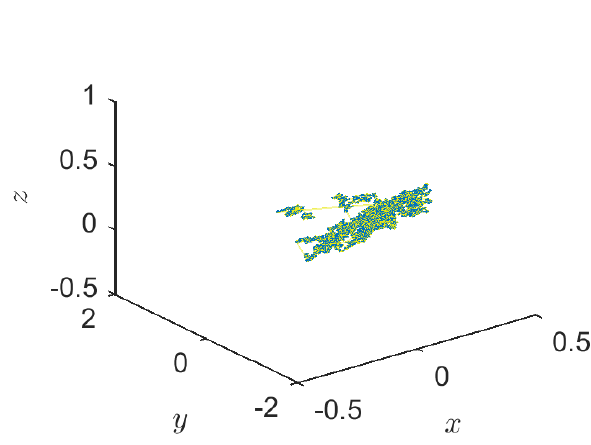}
	\caption{Example sample paths for the $\alpha$-stable Levy process generated by using the exit distribution in Theorem \ref{BGR} for spheres of radius $10^{-6}$. Rows shows sample paths in two- and three-dimensions for $\alpha=0.9$ (left) and $\alpha=1.8$ (right). The yellow lines indicate jumps of the process and blue dots show where the process has been.} \label{fig:sample_path}
\end{figure}
		
On account of  classical Feynman--Kac representation, simulation of solutions to parabolic and elliptic equations involving the fractional Laplacian, and more generally the infinitesimal generator of a L\'evy process are synonymous with the simulation of the paths of the associated stochastic process. On account of the fact that such equations occur naturally in mathematical finance in connection with (exotic) option pricing,  there are already many numerical and stochastic methods in existence for  the general L\'evy setting. The reader is referred, for example, to the books \autocite{Cont, Leven} and the references therein. Other sources offering simulation techniques can be found, e.g.\ \autocite{Polish, Rosinski1, Rosinski2, Rosinski3}. Similarly to works in mathematical finance, they are mostly focused on the approximation of the stable process (and indeed the general L\'evy process) by a compound Poisson process or a power-series representation of the path, with a diffusive component to mimic the effect of small jumps. To our knowledge, however, the walk-on-spheres approach to path simulation has not been used in the context of simulating stable processes to date, nor, as alluded to above, to the end of numerically solving Dirichlet-type problems for the fractional Laplacian. 

\section{Walk-on-spheres for the fractional Laplacian}\label{WoSfL}
		
We start by describing the walk-on-spheres for the fractional-Laplacian Dirichlet problem  (\ref{aDirichlet}) on a convex domain $D$. The domain $D$ may be unbounded, as long as $D^{\textrm{c}}$ has non-zero measure (even though Theorem \ref{corr} requires boundedness).
  Fix $x\in D$. The walk-on-spheres $(\rho_n$, $n\geq 0)$, with $\rho_0 = x$ is defined in a similar way to  the Brownian setting in the sense that, given $\rho_{n-1}$, 
the distribution of $\rho_{n}$ is selected according to an independent copy of  $X_{\sigma_{B_n}}$ under $\mathbb{P}_{\rho_{n-1}}$, where $B_n = \{x\in\mathbb{R}^d\colon |x - \rho_{n-1}|< r_n\}$ and $\sigma_{B_n} = \inf\{t>0: X_t\not\in B_n\}$. The algorithm comes to an end at the random index $N = \min\{n\geq 0\colon \rho_n\not\in D\}$, again using the standard understanding that $\min\emptyset \coloneqq \infty$. 
See for example the depiction in Figure \ref{4stepsonly}.
		 
\begin{figure}[h]
	\centering
	\includegraphics[width=0.7\linewidth]{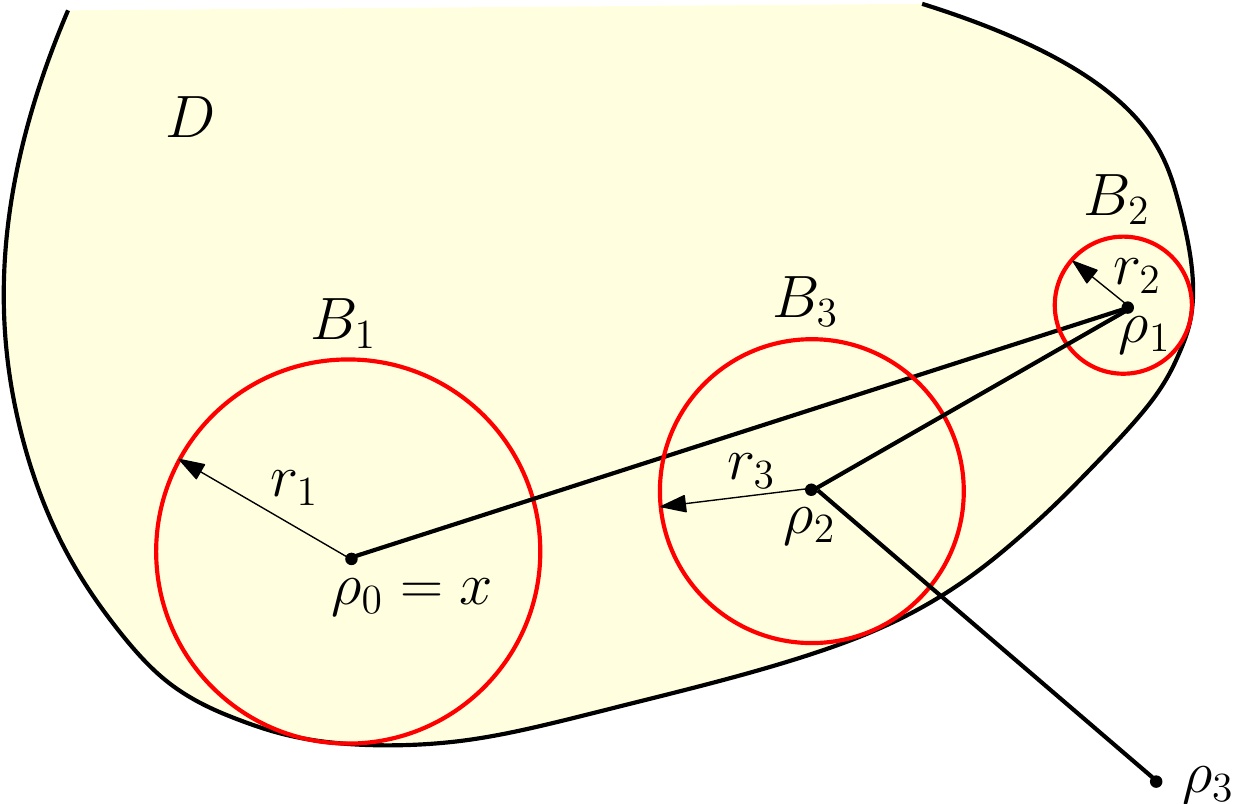}
	\caption{Steps of the walks-on-sphere algorithm until exiting the convex domain $D$ in the stable setting. In this realisation, $N = 3$.}
	\label{4stepsonly}
\end{figure}

Even though the domain $D$ may be unbounded, our main result predicts that, irrespective of the point of issue of the algorithm, there will always be at most a geometrically distributed number of steps (whose parameter also does not depend on the point of issue) before the algorithm ends. 
		
\begin{theo}\label{main} Suppose that $D$ is a convex domain.
	For all $x\in D$, there exists a constant $p = p(\alpha,d)>0$ (independent of $x$ and $D$) and a real-valued random variable $\Gamma$ such that $N\leq \Gamma$ almost surely, where 
	\[
		\mathbb{P}(\Gamma  = k ) = (1-p)^{k-1}p, \qquad k\in\mathbb{N}.
	\]
\end{theo}
\noindent There are a number of remarks that we can make from the conclusion above. 
\begin{itemize}
	\item[1.] Although $\Gamma$ has the same distribution for each $x\in D$, it is not the same random variable for each $x\in D$. As we shall see in the proof of the above theorem, the inequality $N\leq \Gamma$ is derived by comparing each step of the walk-on-spheres algorithm with a sequence of Bernoulli random variables. This sequence of Bernoulli random variables are defined up to null sets which may be different under each  $\mathbb{P}_x$. Therefore, whilst the distribution of $\Gamma$ does not depend on $x$, its null sets do. 
	      	      	      	      	      	      	      	      	      	      	      	      	      	      	      	      	      	      	      	      	      	      	      	      	      	      	      	      	      	      	      	      	      	      	      	      	      	      	      	      	      	      	      	      	      	      	      	      	      	      	      	      	      	      	      	      	      	      	      	      	      	      	      	      	      	      	      	      	      	      	      	      		      		      	      	                                                      
	\item[2.] The stochastic domination in Theorem \ref{main} is much stronger than the usual comparison of the mean number of steps. Indeed, whilst it immediately implies that $\mathbb{E}_x[N] = 1/p$, we can also deduce that there is an exponentially decaying tail in the distribution of the number of steps. Specifically, for any $x\in D$,
	      \[
	      	\mathbb{P}(N >n )%
	      	\leq \mathbb{P}(\Gamma >n )  %
	      	= (1-p)^n, \qquad n\in\mathbb{N}.
	      \]
	\item[3.]  The randomness in the geometric random variables $\Gamma$ is heavily correlated to $N$. The fact that each of the $\Gamma$ are geometrically distributed has the advantage that 
	      \[
	      	\sup_{x\in D}\mathbb{E}_x[N] \leq \sup_{x\in D}\mathbb{E}_x[\Gamma] = \frac{1}{p}.
	      \]
	      However, it is less clear what kind of distributional properties can be said of the random variable 
	      $
	      \sup_{x\in D}\Gamma,
	      $ 
	      which almost surely upper bounds $\sup_{x\in D}N$.
\end{itemize}

Finally, it is worth stating formally that the walk-on-spheres algorithm is unbiased and therefore, providing $\mathbb{E}_x[\bc(X_{\tau_D})]<\infty$, the \lln  applies and a straightforward  Monte Carlo simulation of the solution to \eqref{aDirichlet} is possible. Moreover, providing $\mathbb{E}_x[\bc(X_{\tau_D})^2]<\infty$, the \clt offers the rate of convergence. 
		
\begin{cor}\label{rate1}
	When $D$ is  bounded and convex and $\bc$ is continuous and  in $ L^1_\alpha(D^{\mathrm{c}})$,
	\begin{equation}
		\label{WoSMC2}
		\lim_{n \to\infty} \frac{1}{n}\sum_{i = 1}^n \bc(\rho^{i}_{N^{i}}) = \mathbb{E}_x[\bc(\rho_{N})]=\mathbb{E}_x[\bc(X_{\tau_D})] = u(x),
	\end{equation}
	almost surely where
	$(\rho^{i}_n, n\leq N^{i})$, $i\geq 1$ are \iid copies of the walk-on-spheres with $\rho_0^i = x\in D$, $i\geq 1$ and  $u(x)$ is the solution to (\ref{aDirichlet}).
	Moreover, when 
	\begin{equation}
		\int_{D^\mathrm{c}}\frac{\bc(x)^2}{1+|x|^{\alpha+d}}\,{\rm d}x<\infty,
		\label{f2}
	\end{equation}
	then $\operatorname{Var}(\bc(\rho_N))<\infty$ and, in the sense of weak convergence,  
	\[
		\lim_{n\to\infty}n^{1/2}\left(\frac 1n \sum_{i=1}^n  \bc(\rho^{i}_{N^{i}})- u(x)\right)= \operatorname{Normal}(0,  \operatorname{Var}(\bc(\rho_N))).
	\]
\end{cor}
\begin{proof}
	The first part is a straightforward consequence of the earlier mentioned  \lln and the fact that Theorem \ref{corr} ensures that $\mathbb{E}_x[\bc(\rho_N)]=\mathbb{E}_x[\bc(X_{\tau_D})]<\infty$. For the second part, we need to show that \eqref{f2} implies $\mathbb{E}_x[\bc(\rho_N)^2]=\mathbb{E}_x[\bc(X_{\tau_D})^2]<\infty$. 
	However, if we consider the computation in (\ref{worksforsquared}) of the Appendix, which shows that  $\mathbb{E}_x[\bc(X_{\tau_D})]<\infty$ when $\bc$ is continuous and in $L^1_{\alpha}(D^\mathrm{c})$, then it is easy to see that the same statement holds replacing $\bc$ by $\bc^2$. Under finiteness of the second moment, the \clt completes the proof.
	%
\end{proof}
\medskip

We now  return to the proof of Theorem \ref{main}. Our approach is to break it  into several parts.  For convenience, we shall henceforth write $X^{(x)}=(X^{(x)}(t)\colon t\geq 0)$ to indicate the dependency of $X$ on its initial position $X_0 = x$ (equivalent to writing $(X, \mathbb{P}_x)$). For any $x = (x_1, \dots, x_d)\in\mathbb{R}^d$ such that $x_1>0$, we have $V(x) = \{(z_1, \dots, z_d)\in\mathbb{R}^d \colon z_1>0\}$ for the open half-space containing $x$ and denote its boundary $\partial V(x) = \{(z_1, \dots, z_d)\in\mathbb{R}^d \colon z_1=0\}$. For any Borel set $A\subset\mathbb{R}^d$, we write 
$
\sigma_A = \inf\{t>0 \colon X_t\not\in A\}.
$
We will typically use in place of $A$ the set $V(x)$ as well as $B(x,1) = \{z\in\mathbb{R}^d\colon |z- x|<1\}$, the unit ball centred at $x\in\mathbb{R}^d$.
Finally write ${\rm\bf i} = (1,0, \dots, 0)\in\mathbb{R}^d$.
		
\begin{lem} \label{scaled} Without loss of generality (by appealing to the spatial homogeneity of $X$ which allows us to appropriately choose our coordinate system)
	suppose that  $x= |x|\,{
		\rm\bf i}\in D$ is such that $\partial V(x)$ is a tangent hyperplane to both $D$ and $B_1$. Then $X^{(x)}_{\sigma_{B_1}}$ is equal in distribution to $|x|\, X^{(\rm\bf i)}_{\sigma_{B({\rm\bf i},1)}}$ and  $X^{(x)}_{\sigma_{V(x)}}$ is equal in distribution to $|x|\,X^{(\rm\bf i)}_{\sigma_{V(\mathbf{i})
	}}$.
\end{lem}
		
\begin{proof}
	The scaling property of $X$ ensures that we can write 
	\begin{equation}
		X^{(x)}_{s} = |x|\hat{X}^{(\mathbf{i})}_{|x|^{-\alpha}s}, \qquad s\geq 0,
		\label{scaling}
	\end{equation}
	where $\hat{X}^{(x)}$ is equal in law to $X^{(x)}$. 
	Note that 
	\begin{align}
		\sigma_{B_1} & = \inf\Bp{t> 0\colon {X}^{(x)}(t)\not\in B(x,\abs{x}) }\notag                                                           \\
		             & = |x|^{\alpha}\,\inf\Bp{|x|^{-\alpha}t> 0\colon |x| \hat{X}^{(\mathbf{i})}(|x|^{-\alpha}t)\not\in B( x,\abs{x}) }\notag \\
		             & = |x|^{\alpha}\,\inf\Bp{u> 0\colon  \hat{X}^{(\mathbf{i})}(u)\not\in B({\rm\bf i},1) }\notag                            \\
		             & \eqqcolon |x|^{\alpha}\,\hat{\sigma}_{B({\rm\bf i},1)}.                                                                 
		\label{tscale}
	\end{align}
	It follows that 
	\begin{equation}
		X^{(x)}_{\sigma_{B_1}} = |x| \hat{X}^{(\mathbf{i})}_{|x|^{-\alpha}  |x|^{\alpha} \hat{\sigma}_{B({\rm\bf i},1) }}\,{\buildrel d \over =}\, |x| {X}^{(\mathbf{i})}_{\sigma_{B({\rm\bf i},1)}},
		\label{scaleB1}
	\end{equation}
	as required. The proof of the second claim follows the same steps and is omitted for the sake of brevity.
	\hfill\end{proof}
																																																																														            
	An important consequence of the previous result is the comparison between the first exit from the largest sphere in $D$ centred at $x$ and the first exit from the tangent hyperplane to the latter sphere.  Recall that $B_n = \{z\in\mathbb{R}^d\colon |z - \rho_{n-1}|< r_n\}$ denotes the $n$th sphere.

	\begin{cor}\label{indicators}Suppose that  $x\in D$ is such that $\partial V(x)$ is a tangent hyperplane to both $D$ and $B_1$. 
		Define under $\mathbb{P}_x$ the indicator random variables 
		\[
			{I}_D = \mathbf{1}_{\{X_{\sigma_{B_1}} \not\in D  \}}\quad\text{ and } \quad {I}_V=\mathbf{1}_{\{X_{\sigma_{B_1}} \not\in V(x)  \}}.
		\]
		Then $\mathbb{P}_x(I_D\geq  I_V)= 1$ and, independently of $x\in D$,  $\mathbb{P}_x(I_V = 1)  =p(\alpha,d)$, 
		where
		\begin{align*}
			p(\alpha,d) & \coloneqq \mathbb{P}_{\mathbf i}(X_{\sigma_{B({\rm\bf i},1)}}\not\in V({\rm\bf i}))                                      \\
			            & =\frac{\Gamma(d/2)}{\pi^{(d+2)/2}}\,\sin(\pi\alpha/2)\,\int_{x_1<-1}\left|1- |x|^2\right|^{-\alpha/2}|x|^{-d}\,{\rm d}x, 
		\end{align*}
		which is a number in $(0,1)$.
	\end{cor}
																																																																														            
	\begin{proof}The inequality follows from the inclusion $D\subset V(x)$. The formula for $p(\alpha,d)$ uses the coordinate system  and scaling property of stable processes in Lemma \ref{scaled} as well as the identity for the first exit from a sphere given by Theorem \ref{BGR}.\hfill\
	\end{proof}

	We are now ready to prove our main result. 
																																																																														            
	\begin{proof}[Proof of Theorem \ref{main}]
		Suppose we condition on the previous positions of the walk-on-spheres, $\rho_0,\dots, \rho_{k-1}$ as well as on the event $\{N>k-1\}$. Thanks to stationary and independent increments as well as  isotropy in the law of a stable process, we can always choose a coordinate system, or equivalently reorient $D$ in  such  a way that $\rho_k = |\rho_k|{\rm\bf i}$.
		This has the implication that, with the aforesaid conditioning,  the random variable $\mathbf{1}_{\{N = k\}}$ is independent of $\rho_0,\dots, \rho_{k-1}$ and equal in law to $I_D(\rho_{k-1})$, where we have abused our original notation to indicate the initial position of $X$ in the definition of $I_D$. Similarly, with the same abuse of notation, the event $I_V(\rho_{k-1})$ is independent of $\rho_0,\dots, \rho_{k-1}$ and equal in law to a Bernoulli random variable with probability of success $p = p(\alpha, d)$. In particular, the sequence $I_V(\rho_k)$, $k\geq 0$ is a sequence of Bernoulli trials. That is to say, if we define
		\[
			\Gamma  = \min\{k\geq 1\colon I_V(\rho_k) = 1 \},
		\]
		then it is geometrically distributed with parameter $p$.
		Thanks to Corollary \ref{indicators}, we also have that  $\mathbb{P}_x(I_D\geq I_V)|_{x = \rho_k}=1$, $k< N$,  that is to say,  $\{I_V(\rho_k)=1\}$ almost surely implies $\{I_D(\rho_k)=1\}$, for $k<N$, and hence
		\[
			\min\{n\geq 1\colon I_D(\rho_k) = 1 \}\leq  \min\{n\geq 1\colon I_V(\rho_k) = 1 \}
		\]
		almost surely.
		In other words, we have $N\leq \Gamma$, almost surely, as required. 
	\end{proof}

	\section{Non-convex domains}
	\label{non_convex}

	The key element in the proof of Theorem \ref{main} is the comparison of the event that the next step of the walk-on-spheres exits the domain $D$ with the event that the next step of the walk-on-spheres exits a larger, more regular domain. More precisely, the aforesaid regular domain  is taken to be the half-space that contains $D$ with boundary hyperplane that is tangent to both the current maximal sphere and $D$. It is the use of a half-space that allows us to work with unbounded domains but which forces the assumption that $D$ is convex. With a little more care, we can remove the need for convexity without disturbing the main idea of the proof. However, this will come at the cost of insisting that $D$ is bounded. It does however, open the possibility that $D$ is not a connected domain.  We give two results in this respect.
																																																																														            
	For the first one, we introduce the following definition, which  has previously been used in the potential analysis of stable processes; see for example \autocite{Chen-Song}.
																																																																														            
	\begin{defin}\rm A domain $D$ in $\mathbb{R}^d$ is said to satisfy the {\it uniform exterior-cone condition}, henceforth written UECC, if there exist constants $\eta > 0$, $r > 0$ and a cone 
		\[
			{\rm Cone}(\eta) = \{x = (x_1,\dots,x_d) \in
			\mathbb{R}^d\colon |x|<\eta x_1\}
		\] such 
		that, for every $z\in \partial D$, there is a
		cone $C_z$ with vertex $z$, isometric to ${\rm Cone}(\eta)$ satisfying $C_z \cap B(z,r) \subset D^{\mathrm{c}}$.
	\end{defin}
	\noindent It is well known that, for example,  bounded $C^{1,1}$ domains satisfy (UECC). We need a slightly more restrictive class of domains than those respecting UECC.
																																																																														            
	\begin{defin}\rm 
		We say that $D$ satisfies the {\it regularised uniform exterior-cone condition}, written RUECC, if it is UECC and the following additional condition holds: for each $x\in D$, suppose that $\partial(x)$ is a closest point on the boundary of $D$ to $x$. Then the isometric cone that qualifies $D$ as UECC can be placed with its vertex at $\partial(x)$ and symmetrically oriented around the line that passes through $x$ and $\partial(x)$.
	\end{defin}
																																																																														            
	\begin{figure}[h]
		\centering
		\includegraphics[width=0.8\linewidth]{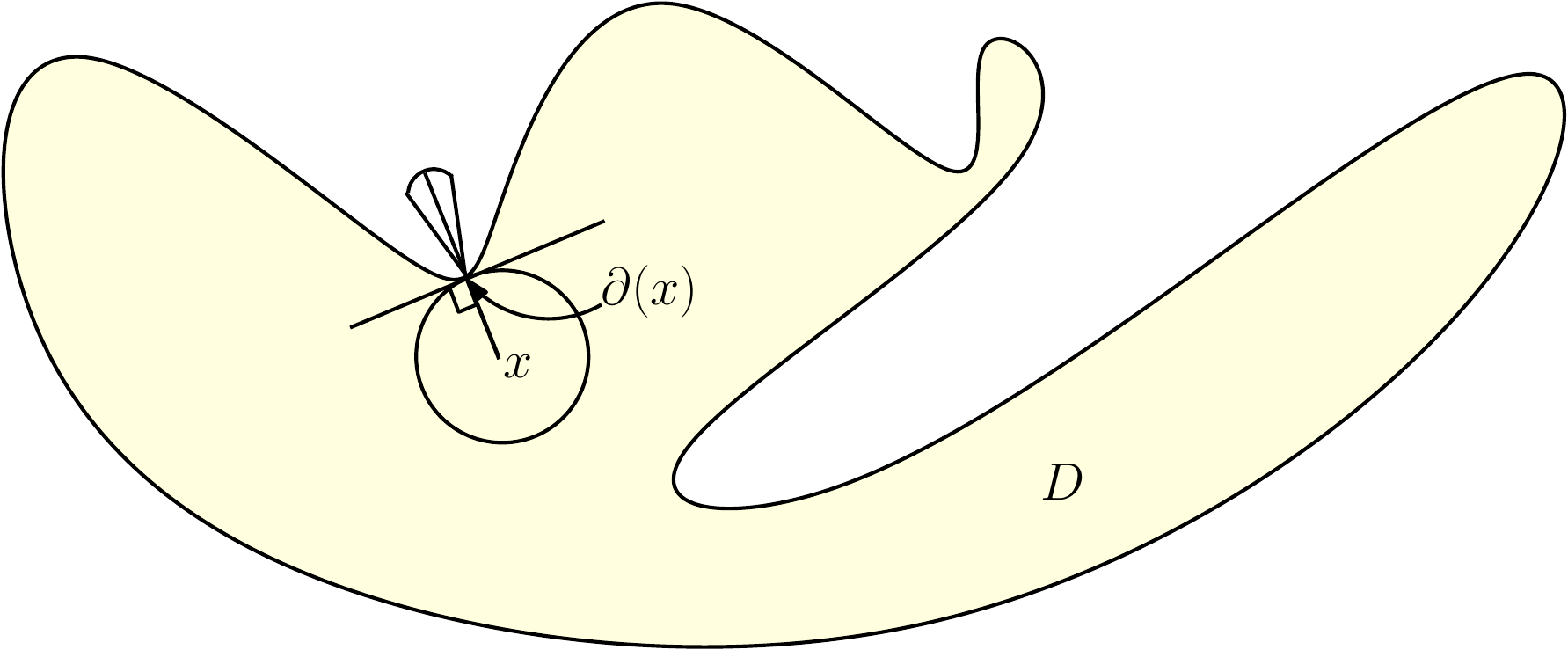}
		\caption{A domain that satisfies the regularised uniform exterior-cone condition}
		\label{fig:class_proof}
	\end{figure}

	\begin{theo}
		Suppose that $D$ is open and bounded (but not necessarily connected) and satisfies RUECC. Then, for each $x\in D$, there exists a random variable $\hat{\Gamma}$ such that $N\leq \hat{\Gamma}$ almost surely and 
		\[
			\mathbb{P}(\hat{\Gamma} = k)  = (1-\hat{q})^{k-1}\hat{q}, \qquad k\in \mathbb{N}, 
		\]
		for some $\hat{q}=\hat{q}(\alpha, D)$.
	\end{theo}
	\begin{proof}Reviewing the proof of Theorem \ref{main}, we note that it suffices to prove that, in the context of Corollary \ref{indicators}, for each $x\in D$, there exists a Bernoulli random variable $\hat{J}_x$ with parameter $\hat{q}$ (independent of $x$) such that $\mathbb{P}_x(I_D\geq \hat{J}_x)=1$. To this end,  we recall that, without loss of generality, we may choose our coordinate system such that   $x = |x|{\rm\bf i}\in D$ is such that $\partial(x) = 0$. The assumption that $D$ is bounded implies that there exists a $\eta$ such that $|x|\leq \eta$. From the definition of RUECC, we know that there exists an $r>0$ and a cone,  $C_{0}$, with vertex at $0$, a closest point on $\partial D$ to $x$, which is symmetrically oriented around the line passing through $x$ and $0$, such that $C_{0, r}\coloneqq C_{0} \cap B(0,r)\subset D^{\texttt{c}}$.
		We have 
		\begin{align*}
			\mathbb{P}_{x}(X_{\sigma_{B_1}} \in C_{0, r})
			  & =\mathbb{P}_{\mathbf i}(X_{\sigma_{B(\mathbf{i},1)}}\in C_{0, r/|x|})                                                                           \\
			  & \geq \mathbb{P}_{\mathbf i}(X_{\sigma_{B(\mathbf{i},1)}}\in C_{0, r/\eta})                                                                    \\
			  & =  \frac{\Gamma(d/2)}{\pi^{(d+2)/2}}\,\sin(\pi\alpha/2)\int_{C_{-{\rm\bf i}, (r/\eta)}}\left|1- |y|^2\right|^{-\alpha/2}|y|^{-d}\,{\rm d}y \\
			  & \eqqcolon\hat{q},                                                                                                                               
		\end{align*}
		where $C_{z, u} \coloneqq  [C_{0} \cap B(0,u)]-\{z\}$, for $z\in\mathbb{R}^d$ and $u>0$. Note that $\hat{q}$ is necessarily strictly positive. Taking account of scaling, we have $\mathbb{P}_x$-almost surely that
		\[
			I_D\geq \mathbf{1}_{\{|x|^{-1}X^{(x)}_{\sigma_{B_1}} \in C_{0, r/|x|}  \}}\geq\mathbf{1}_{\{|x|^{-1}X^{(x)}_{\sigma_{B_1}} \in C_{0, r/\eta}  \}}\eqqcolon \hat{J},
		\]
		where $\hat{J}$ is a Bernoulli random variable with parameter $\hat{q}$. Stochastic dominance, $N\leq \hat{\Gamma}$ almost surely, follows by the same line of reasoning as in the proof of Theorem \ref{main}.
	\end{proof}
																																																																														            
	For the second result, we completely relax the geometrical requirements on $D$ at the expense of efficiency.
	With an abuse of our earlier notation, we introduce
	\[
		N(\varepsilon) = \min\Bigl\{n\geq 0\colon  \rho_n\not\in D \text{ or }\inf_{z\in \partial D}|\rho_n-z|<\varepsilon \Bigr\}.
	\]
	Intuitively,  $N(\varepsilon)$ is the step that exits the inner $\varepsilon$-thickened boundary of $D$.
																																																																														            
	\begin{theo} Suppose that $D$ is open and bounded (but not necessarily connected).
		Then for all $x\in D$, there exists a constant $q_\varepsilon = q_\varepsilon(\alpha,D)>0$ (independent of $x$) and a random variable $\Gamma^\varepsilon$ such that $N\leq \Gamma^\varepsilon$ almost surely, where 
		\[
			\mathbb{P}_x(\Gamma^\varepsilon  = k ) = (1-q_\varepsilon)^{k-1}q_\varepsilon, \qquad k\in\mathbb{N}.
		\]
		Moreover,  $q_\varepsilon = \mathcal{O}(\varepsilon^\alpha)$ as $\varepsilon\downarrow 0$. In particular 
		\begin{equation}
			\mathbb{E}_x[N(\varepsilon)]  = \mathcal{O}(\varepsilon^{-\alpha}),\qquad \text{as $\varepsilon \downarrow 0$.}
			\label{ea}
		\end{equation}
	\end{theo}
	\begin{proof}
		Define
		\[
			\delta\coloneqq \inf\Bp{r>0\colon  D\subset B(x,r) \text{ for all } x\in D }, 
		\]  
		so that any sphere of radius $\delta$ centred at $x\in D$  contains $D$.
		Once again, we recall that, without loss of generality, we may choose our coordinate system such that   $x = |x|\,{\rm\bf i}\in D$ is such that $\partial V(x)$ is a tangent hyperplane to  $B_1$  and such that $0\in \partial B_1 \cap \partial V(x)\cap\partial D$. 
		Then, taking account of scaling, and that, for all $x\in D$ such that $\inf_{z\in\partial D}|x-z|\geq \varepsilon$, with the particular choice of coordinates described above, $\delta/|x|\leq \delta/\varepsilon$, we have 
		\[
			\mathbf{1}_{\{N(\varepsilon) = 1\}}\geq \mathbf{1}_{\{X^{(x)}_{\sigma_{B_1}} \not\in B(x,\delta)  \}} =  \mathbf{1}_{\{|x|^{-1}X^{(x)}_{\sigma_{B_1}} \not\in B({\rm\bf i},\delta/|x|)  \}} \geq  \mathbf{1}_{\{|x|^{-1}X^{(x)}_{\sigma_{B_1}} \not\in B({\rm\bf i},\delta/\varepsilon)  \}}.  
		\]
		Recall, however, from (\ref{scaleB1}) that $|x|^{-1}X^{(x)}_{\sigma_{B_1}} {\buildrel d \over =} {X}^{(\mathbf{i})}_{\sigma_{B({\rm\bf i},1)}}$. It therefore follows that, $\mathbb{P}_x$-almost surely,
		\[
			\mathbf{1}_{\{N(\varepsilon) = 1\}}
			\geq \mathbf{1}_{\{X^{(\mathbf{i})}_{\sigma_{B({\rm\bf i}, 1)}} \not\in B(\mathbf{i},\delta/\varepsilon)  \}} \eqqcolon J^\varepsilon,
		\]
		where $J^\varepsilon$ is a Bernoulli random variable with parameter 
		\[
			q_\varepsilon (\alpha, D)%
			=  \mathbb{P}_{\mathbf i}(X_{\sigma_{B({\rm\bf i},1)}}\notin B({\rm\bf i},\delta/\varepsilon))%
			=\frac{\Gamma(d/2)}{\pi^{(d+2)/2}}\,%
			\sin(\pi\alpha/2)\,%
			\int_{|y|\geq \delta/\varepsilon}\left|1- |y|^2\right|^{-\alpha/2}|y|^{-d}\,{\rm d}y.
		\] 
		Reverting to generalised spherical polar coordinates, in particular recalling that the Jacobian with respect to Cartesian coordinates is no larger than $|x|^{d-1}$ (see \cite{Blum}), we can estimate 
		\[
			q_\varepsilon (\alpha, D)%
			\leq \frac{\Gamma(d/2)}{\pi^{(d+2)/2}} \, \sin(\pi\alpha/2)\,\int_{\delta/\varepsilon}^\infty r^{-(\alpha +1)}dr%
			=\mathcal{O} (\varepsilon^\alpha).
		\]
		Reviewing the line of reasoning in the proof of Theorem \ref{main}, we see that this comparison of events on the first step can be repeated at each surviving step of the algorithm to deduce the claimed result.
	\end{proof}
	The $\mathcal{O}(\varepsilon^{-\alpha})$ bound in \eqref{ea} can be compared  with the bounds achieved by \autocite{BB} for the classical walk-on-spheres with Brownian motion for domains with more general geometries than convex.  The worst case in \autocite{BB} is  $\mathcal{O}(\varepsilon^{2-4/a})$ for a parameter $a>0$ (describing the domain's thickness or fractal boundary). Notably in the limit $\alpha\to 2$ ($X$ converges to Brownian motion) and $a \to \infty$ (the domain loses regularity), the two agree with an $\mathcal{O}(\varepsilon^{-2})$ bound.
	\section{Fractional Poisson problem}\label{inhomogenous}
																																																																														            
	We are now interested in using the walk-on-spheres process to find the solution to the inhomogeneous version of (\ref{aDirichlet}), namely
	\begin{gather}
		\begin{aligned}
			-(-\Delta )^{\alpha/2}u(x) & =-\rhs(x), & \qquad x & \in D,          \\
			u(x)                       & = \bc(x), & x        & \in  D^{\rm c}, 
		\end{aligned} 
		\label{aDirichlet_g}
	\end{gather}
	for suitably regular functions $\rhs\colon D\to \mathbb{R}$ and $\bc\colon D^{\rm c} \to \mathbb R$. We want to identify a Feynman--Kac representation for solutions to \eqref{aDirichlet_g} for suitable assumptions on $\bc, \rhs$ and $D$. Throughout this section, we adopt the setting of the following theorem.
																																																																														            
	\begin{theo}\label{hasacorr} Let $d\geq 2$ and assume that $D$ is a  bounded domain in $\mathbb{R}^d$.
		Suppose that  $\bc$   is a  continuous function  which belongs to $L^1_\alpha(D^\mathrm{c})$. Moreover, suppose that $\rhs$ is a function in $ C^{\alpha +\varepsilon}(\overline{D})$ for some $\varepsilon>0$. 
		Then there exists a  unique  continuous solution to   \eqref{aDirichlet_g} in $L^1_\alpha(\mathbb{R}^d)$ which  is given by
		\begin{equation}
			u(x) = \mathbb{E}_x[\bc(X_{\sigma_D})] + \mathbb{E}_x\left[\int_0^{\sigma_D} \rhs(X_s)\,{\rm d}s\right], \qquad x\in D,
			\label{non_homg_FK}
		\end{equation}
		where $\sigma_D=\inf\{t>0\colon X_t\not\in D\}$.
	\end{theo}
	\noindent The combinations of Theorem 2.10 and 3.2 in   \autocite{bucur} treat the case that $D$ is a ball. In the more general setting, amongst others, \autocite{B99},  \autocite{R-O1} and \autocite{R-O2} (see also citations therein) offer results in this direction, albeit from a more analytical perspective. We give a new probabilistic proof of Theorem \ref{hasacorr} in the Appendix using a method that combines the idea of walks-on-spheres with the version of Theorem \ref{hasacorr} when $D$ is a ball. It is for this reason that the (otherwise unclear) need for the assumption that $\rhs\in C^{\alpha +\varepsilon}(\overline{D})$ enters. Note in particular that Theorem \ref{corr} follows as a corollary.

	We can develop the expression in \eqref{non_homg_FK} in terms of the walk-on-spheres $(\rho_n, n\leq N)$, providing the basis for a Monte Carlo simulation. What will work to our advantage here is another explicit identity that appears in \autocite{BGR}.  Define
	\begin{align*}
		V_r(x,{\rm d}y) & \coloneqq \int_0^\infty \mathbb{P}_x(X_t\in {\rm d}y, \, t<{\sigma_{B(x,r)}}  )\,{\rm d}t, \qquad x\in\mathbb{R}^d, \;|y|<1,\; r>0. 
	\end{align*}
																																																																														            
	\begin{theo}[Blumenthal, Getoor, Ray 1961]  The expected occupation measure of the stable process prior to exiting a unit ball centred at the origin is given, for $|y|<1$, by
		\begin{align}
			V_1(0,{\rm d}y)= 2^{-\alpha}\,\pi^{-d/2}\, \frac{\Gamma(d/2)}{\Gamma(\alpha/2)^{2}}\,            
			|y|^{\alpha -d}\, \pp{\int_0^{|y|^{-2}-1}(u+1)^{-d/2}u^{\alpha/2-1}{\rm d}u}\,{\rm d}y.\label{V} 
		\end{align}
		\end{theo}
																																																																														            
	\noindent Whilst the above identity is presented in a probabilistic context, it has a much older history in the analysis literature. Known as Boggio's formula, the original derivation in the setting of potential theory  dates back to \autocite{Bog}. See the discussion in \autocite{DR,bucur}.
																																																																														            
	In the next result, we will write as a slight abuse of notation  $V_r(x,\rhs(\cdot)) = \int_{|y-x|<r}\rhs(y)\,V_r(x,{\rm d}y)$ for bounded measurable $\rhs$.
																																																																														            
	\begin{lem}\label{integral}For $x\in D$, $\bc\in L^1_\alpha(D^\mathrm{c})$ and $\rhs\in C^{\alpha +\varepsilon}(\overline{D})$, we have the representation 
		\[
			u(x) =\mathbb{E}_x[\bc(\rho_{N})] + \mathbb{E}_x\left[\sum_{n=0}^{N-1} r_n^{\alpha} V_{1}(0, \rhs(\rho_n + r_n\cdot))\right].
		\]
	\end{lem}
	\begin{proof}
		Given the walk-on-spheres $(\rho_n, n\leq N)$ with $\rho_0 = x\in D$,  define $\sigma_n$ jointly with $\rho_n$ so that, given $\rho_{n-1}$,  $(\rho_n, \sigma_n)$ is equal in law to $(X_{\sigma_{B_n}}, \sigma_{B_n})$ under $\mathbb{P}_{\rho_{n-1}}$.
		We can now  represent the second  expectation on the right-hand side of (\ref{non_homg_FK}) in the form 
		\begin{equation}
			\mathbb{E}_x\left[\sum_{n\geq 0} \mathbf{1}_{\{\rho_n\in D\}} \int_{0}^{\sigma_{n+1}} \rhs\left(\rho_n + X^{(n+1)}_s\right)\,\mathrm{d}s\right],
			\qquad x\in D,
			\label{withrho}
		\end{equation}
		where $X^{(n)}$ are independent copies of $(X, \mathbb{P}_0)$. 
		Applying Fubini's theorem, then conditioning each expectation on $\mathcal{F}_{n}\coloneqq \sigma(\rho_k\colon k\leq n)$ followed by Fubini's theorem again, we have 
		\begin{align*}
			\mathbb{E}_x\left[\int_0^{\sigma_D} \rhs(X_s)\,\mathrm {d}s\right] & =\sum_{n\geq 0}\mathbb{E}_x\left[ \mathbf{1}_{\{\rho_n\in D\}} \left. 
			\mathbb{E}_{y}\left[\int_{0}^{\sigma_{B(y,r)}} \rhs( X_s)\,\mathrm{d}s\right] 
			\right|_{y = \rho_n, r = r_n}\right]\\
			& =\sum_{n\geq 0}\mathbb{E}_x\left[ \mathbf{1}_{\{\rho_n\in D\}}        
			V_{r_n}(\rho_n, \rhs(\cdot))
			\right]\\
			& =\mathbb{E}_x\left[\sum_{n= 0}^{N-1} 
			V_{r_n}(\rho_n, \rhs(\cdot))
			\right].
		\end{align*}
		The proof is completed once we show that $V_r(x,g) = r^{\alpha}V_1(0, \rhs(x + r\cdot)),$ for $r>0$, $x\in\mathbb{R}^d$ and bounded measurable $\rhs$.
		To this end, we appeal to spatial homogeneity and the, now, familiar computations using the scaling property of stable processes:
		\begin{align}
			V_r(x,\rhs(\cdot)) & = \mathbb{E}_x\left[\int_0^{\sigma_{B(x,r)}} \rhs(X_t) \,{\rm d}t\right]\nonumber                   \\
			         & = \mathbb{E}_0\left[\int_0^{\sigma_{B(0,r)}} \rhs(x+X_t) \,{\rm d}t\right]\nonumber                 \\
			         & = \mathbb{E}_0\left[\int_0^{\sigma_{B(0,1)}} r^{\alpha}\,\rhs(x+ r\,X_s) \,{\rm d}s\right]\nonumber \\
			         & =\int_{|y|<1} r^{\alpha} \rhs(x + r\,y)\, V_1({0,\rm d}y)\nonumber                                      \\
			         & =r^\alpha \,V_1(0, \rhs(x + r\,\cdot)).                                                             
			\label{timescale}
		\end{align}
		The proof is now complete.
	\end{proof}

	Lemma \ref{integral} now informs a Monte Carlo procedure based on simulating the quantity
	\[
		\chi \coloneqq \bc(\rho_{N}) + \sum_{n=0}^{N-1} r_n^{\alpha}\, V_{1}(0, \rhs(\rho_n + r_n\cdot))
		, \qquad x\in D,
	\]
	which is again justified by an obvious \lln and the \clt in the spirit of Corollary \ref{rate1}.
	\begin{cor}\label{rate2}
		When $D$ is  bounded and convex, $\bc$ is continuous and  in $ L^1_\alpha(D^\mathrm{c})$ and $\rhs$ is a function in $ C^{\alpha +\varepsilon}(\overline{D})$ for some $\varepsilon>0$, then
		\begin{equation}
			\label{WoSMC3}
			\lim_{n \to\infty} \frac{1}{n}\sum_{i = 1}^n \chi^{i} = \mathbb{E}_x\left[ \bc(\rho_{N}) + \sum_{n=0}^{N-1} r_n^{\alpha}\, V_{1}(0, \rhs(\rho_n + r_n\cdot))\right] 
			= u(x),
		\end{equation}
		almost surely where
		$\chi^{i}$, $i\geq 1$ are \iid copies of $\chi$ and  $u(x)$ is the solution to (\ref{aDirichlet_g}).
		Moreover, when 
		\begin{equation}
			\int_{D^\mathrm{c}}\frac{\bc(x)^2}{1+|x|^{\alpha + d}}\,{\rm d}x<\infty.
			\label{f22}
		\end{equation}
		then $\operatorname{Var}(\chi)<\infty$ and, in the sense of weak convergence,  
		\[
			\lim_{n\to\infty}n^{1/2}\left(\frac 1n \sum_{i=1}^n  \chi^{i}- u(x)\right)= \operatorname{Normal}(0,  \operatorname{Var}(\chi)).
		\]
	\end{cor}
	\begin{proof}
		Theorem \ref{hasacorr} and Lemma \ref{integral} ensure that the \lln may be invoked. For the \clt, we need $\mathbb{E}_x[\chi^2]<\infty$. Taking account of the fact that $\chi$ is the sum of two terms, the Cauchy--Schwarz inequality ensures that $\mathbb{E}_x[\chi^2] $ is finite if 
		$\mathbb{E}_x\left[\bc(\rho_{N})^2\right] $ and $ \mathbb{E}_x\left[\left( \sum_{n=0}^{N-1} r_n^{\alpha}\, V_{1}(0, \rhs(\rho_n + r_n\cdot))\right)^2\right]$ are finite. Recall that  $\mathbb{E}_x\left[\bc(\rho_{N})^2\right] = \mathbb{E}_x[\bc(X_{\sigma_D})^2]$ and, from Corollary \ref{rate1},  that \eqref{f22} is sufficient to ensure that this expectation is bounded.
																																																																																																																																																										                        
		Now note that, on account of the fact that $\rhs$ is bounded, there exists a constant $\kappa\in(0,1)$, such that, for each $n\leq N$, appealing to (\ref{timescale}), we have $r_n^{\alpha}\, V_{1}(0, \rhs(\rho_n + r_n\cdot))\leq \kappa \sigma_n$, where $\sigma_n$ is the time it takes for the walk-on-spheres to exit the $n$th sphere.  Thus 
		$\sum_{n=0}^{N-1} r_n^{\alpha}\, V_{1}(0, \rhs(\rho_n + r_n\cdot))\leq\kappa  \sum_{n=0}^{N-1} \sigma_n = \kappa\,\sigma_D$.
		We thus have that 
		\[
			\mathbb{E}_x\left[\left( \sum_{n=0}^{N-1} r_n^{\alpha}\, V_{1}(0, \rhs(\rho_n + r_n\cdot))\right)^2\right]\leq \kappa^2 \mathbb{E}_x[\sigma_D^2].
		\]
		However, the latter expectation can be bounded by $\mathbb{E}_x[\sigma_{B^*}^2]$, where $B^* = B(x, R)$ for some suitably large $R$ such that $D$ is compactly embedded in $B^*$. Moreover, appealing to \autocite{GET}, we know that  $\mathbb{E}_x[\sigma_{B^*}^2]$ is bounded. 
	\end{proof}

	\section{Numerical experiments}\label{numerics}
	In the following section, all of the routines associated with the simulations are publicly available at the following repository: 
																																																																														            
	\medskip
																																																																														            
	\url{https://bitbucket.org/wos_paper/wos_repo}

	\medskip
																																																																														            
	\noindent For the Monte Carlo procedure, independent copies of the walk-on-spheres $(\rho_n, n\leq N)$ need to be simulated whereby, by the Markov property, every new point in the sequence can be expressed as $\rho_{n+1}=\rho_n+X'_{\sigma'_{B(0,{r_n})} }$, where $X'$ is an independent version of $X$ and $$\sigma'_{B(0,{r_n})}  = \inf\{t>0\colon X'_t \not\in B(0,{r_n})\}.$$ In other words, $\rho_{n+1}$ is an exit point from a ball $B(0,{r_n})$ under $\mathbb{P}_0$ translated by $\rho_n$. A consequence of  Lemma \ref{BGR} is that the exit distribution of $X'_t$ from $B(0,r_n)$, $r_n>0$,  can be, via a change of variable $y = \tilde{y}/r_n$, written as 
	\begin{equation}
		\mathbb{P}_0(X_{\sigma_{B(0,r_n)}}\in {\rm d}\tilde{y}) = \pi^{-(d/2+1)}\Gamma(d/2)\sin(\pi\alpha/2)\left|r_n^2-|\tilde{y}|^2\right|^{-\alpha/2}|\tilde{y}|^{-d} r_n^{\alpha}\,{\rm d}\tilde{y}, \qquad  |\tilde{y}|>r_n.
		\label{rdy}
	\end{equation}
	For $d=2$, it is more convenient to work with polar coordinates   $(r,\theta)$ in order to separate variables in \eqref{rdy}. Indeed, recalling that  ${\rm d}\tilde{y}=r\,{\rm d}r\,{\rm d}\theta$, we have
	\begin{align}
		\mathbb{P}_0(X_{\sigma_{B(0,r_n)}}\in {\rm d}\tilde{y}) & = \frac{2}{\pi}\sin(\pi\alpha/2)\left(r^2-r_n^2\right)^{-\alpha/2} r_n^{\alpha}\,\frac{{\rm d}r}{r} \times \frac{{\rm d}\theta}{2\pi}\,,\qquad r>r_n.\label{DISPOL} 
	\end{align}
	From \eqref{DISPOL}, we see that the angle $\theta$ is sampled uniformly on $[0,2\pi]$ whereas we can sample the radius $r$ via the inverse-transform sampling method. 
	To this end, noting that $\sin(\pi\alpha/2)B(\alpha/2, 1-\alpha/2)=\pi$, the first factor on the right-hand side of \eqref{DISPOL} is the density of a distribution with cumulative distribution function $F$. The inverse of $F$  can be identified as follows:
	For $x\in[0,1]$, 
	\begin{align*}
		F^{-1}(x)=r_n\left(I^{-1}(
		1-x;\alpha/2,1-\alpha/2))\right)^{-1/2}, 
	\end{align*}
	where $I^{-1}(x;z,w)$ is the inverse of the incomplete beta function 
	\[
		I(x;z,w)\coloneqq\frac{1}{B(z,w)}\int_{0}^{x}u^{z-1}(1-u)^{w-1}\,{\rm d}u, \qquad x\in [0,1],
	\]
	and $B(z,w)\coloneqq \int_{0}^{1}u^{z-1}(1-u)^{w-1}\,{\rm d}u$ is the beta function. 
																																																																														             
	The homogeneous part of the solution to \eqref{aDirichlet_g} is somewhat easier to compute than  the inhomogeneous part, which additionally involves  numerical computation of  the integral $r_n^{\alpha}V_1(0,\rhs(\rho_n+r_n\cdot))$ in \eqref{WoSMC3}. 
	To develop this expression, we use the substitution $u=(1-t)/t$ for the integral in \eqref{V} and hence, when $d = 2$,
	for $|y|<1$,
	\begin{equation} 
		V_1(0,{\rm d}y)=c_{2,\alpha}B(1-\alpha/2,\alpha/2)|y|^{\alpha-2}(1-I(|y|^2;1-\alpha/2,\alpha/2))
	\end{equation}
	with $c_{2,\alpha}=2^{-\alpha}\pi^{-1}\Gamma(\alpha/2)^{-2}$. Moreover, by converting to polar coordinates $(r,\theta)$, the simulated quantity at step $n$ becomes
	\begin{align*}
		  & r_n^{\alpha}V_1(0,\rhs(\rho_n+r_n\cdot))                                                                                        \\
		  & =r_n^{\alpha}c_{2,\alpha}B(1-\alpha/2,\alpha/2) \int_{|y|<1}                                                                             
		\rhs(\rho_n+r_ny)|y|^{\alpha-2}\left(1-I(|y|^2;1-\alpha/2,\alpha/2)\right)\,{\rm d}y\\
		  & =r_n^{\alpha} c_{2,\alpha}B(1-\alpha/2,\alpha/2)2\pi \alpha^{-1}\int_{0}^{1}\int_{-\pi}^{\pi} \rhs(\rho_n+r_n r(\cos\theta,\sin\theta)) \\
		  & \hspace{7cm}\times\left(1-I(r^2;1-\alpha/2,\alpha/2)\right)\frac{{\rm d}\theta}{2\pi}\times \alpha r^{\alpha-1}\,{\rm d}r.   
	\end{align*}
	We used the Monte Carlo approach for evaluating this integral. Consider independent random variables $\Theta\sim U(-\pi,\pi)$ and $R=X^{1/\alpha}$ such that $X\sim U(0,1)$. Then $R$ has the probability density function $f_R(r)=\alpha r^{\alpha-1}$ and we want to evaluate
	\begin{equation}
		r_n^{\alpha}V_1(0,\rhs(\rho_n+r_n\cdot))=a_{2,\alpha} r_n^\alpha \mathbb{E}\left[\left(1-I(R^2;1-\alpha/2,\alpha/2)\right)\rhs(\rho_n+r_n R(\cos\Theta,\sin\Theta))\right] \label{MCEXP}
	\end{equation}
	with $a_{2,\alpha}=\alpha^{-1}2^{-\alpha+1}\Gamma(\alpha/2)^{-2}B(1-\alpha/2,\alpha/2)$. We simulate $n_{R,\Theta}$ samples of pairs $(R,\Theta)$ and compute the sample mean of the quantity in \eqref{MCEXP}. The quantity is evaluated more efficiently by writing \begin{equation}
	\rhs(\rho_n+r_n R(\cos\Theta,\sin\Theta))= \rhs(\rho_n)+[\rhs(\rho_n+r_n R(\cos\Theta,\sin\Theta))-\rhs(\rho_n)]
	\label{[]}
	\end{equation}
	This gives two terms: one can be evaluated directly (by storing $\mathbb{E} [1-I(R^2; 1-\alpha/2, \alpha/2)]$) and the second can be evaluated using a Monte Carlo method, but with smaller variance (as the quantity in square brackets in \eqref{[]} is $\order{r_n}$). It is worth noting that a similar mixed approach using the trapezoidal rule over $\theta$ and randomising $r$ as earlier  for evaluating the left-hand side of  \eqref{MCEXP}  was also tested. However, results showed that the pure Monte Carlo approach is, in comparison to the mixed one, superior with regards to accuracy and computational cost. With this view, we decided to focus on the first one.
																																																																														            
	Accuracy of this algorithm and its feasibility of implementation was checked with model solutions to problems of the type \eqref{aDirichlet_g} and they are presented below in order. 
	%
	%
	%
	%
	%
	%

	\subsection{Free-space Green's function} 
	The free-space Green's function for the fractional Laplacian $(-\Delta)^{\alpha/2}$   is
	\[
		G(x, y)=c_{d,\alpha} \frac{1}{|x-y|^{\alpha-d}}
	\]
	for a constant $c_{d,\alpha}$ for $d>1$ and $\alpha\in(0,2$); see \autocite{bucur}. 
	If the  point $y$ is chosen outside a domain $D$, then we can construct $G$ as an exact solution to the homogeneous version of the fractional Dirichlet problem in
	\eqref{aDirichlet}; that is,  $u(x)=G(x,y)$ for $x\in D$ and $\bc(x)=G(x,y)$ for $x\not\in D$.
	Figure \ref{fig:test4} shows the results of applying the walk-on-spheres algorithm to evaluate $u(0.6, 0.6)$ with  $10^6$ samples, where $D$ is a unit ball in $\mathbb{R}^2$ centred at the origin and $y=(2,0)$. We observe the samples $\bc(\rho_{N})$ have larger variance when $\alpha$ is small and a larger error results from the same number of samples.  
	\begin{figure}
		\centering
		\includegraphics{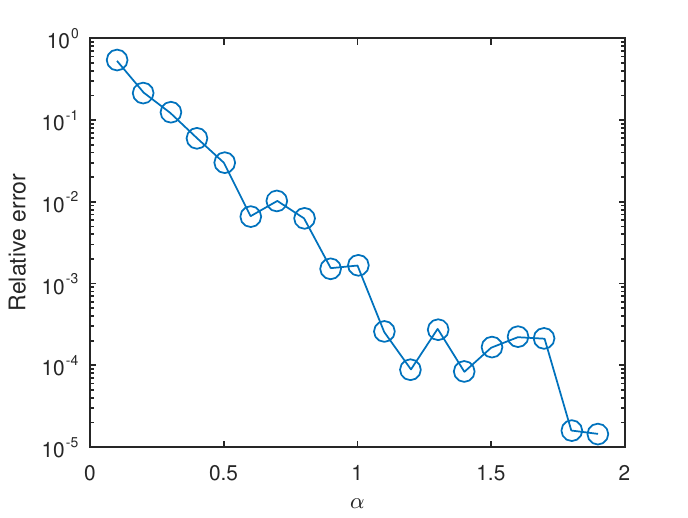}\quad
		\includegraphics{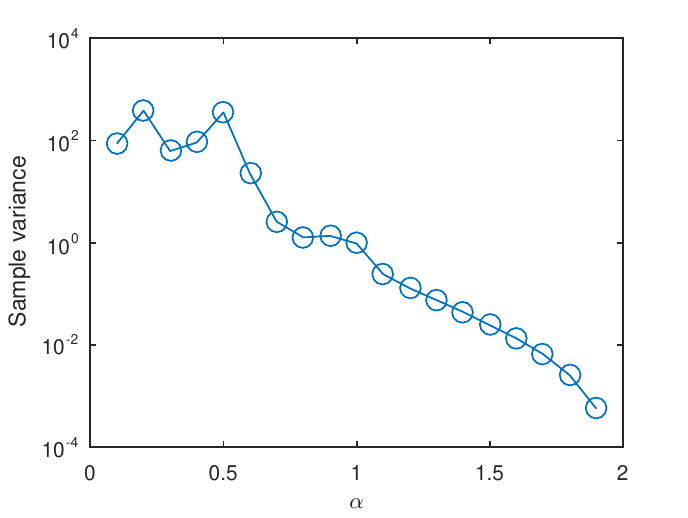}\\
		\caption{Example simulation for \eqref{aDirichlet} with exterior data $\bc(x)=G(x,y)$ with $y=(2,0)$ on the domain given by the unit ball centred at the origin,  based on $10^6$ samples.  The left-hand plot shows the relative error and the right-hand plot shows the sample variance. The sample variance is larger for small $\alpha$ as the process stops further away from the boundary and can see the singularity at $(2,0)$ in the exterior data. Accordingly, the relative error is higher as we are using a fixed number of samples.} \label{fig:test4}
	\end{figure}
																																																																														            
	\subsection{Gaussian data} 
																																																																														            
	For the Poisson problem \eqref{aDirichlet_g}, we take $D$ to be the unit ball in $\mathbb{R}^2$,  exterior data
	\[\bc(x)=\exp(-| x-y|^2), \qquad x\in D^{\rm c},
	\]
	for a given $y\in \mathbb{R}^2$, and zero source term $\rhs=0$. We can represent the solution to \eqref{aDirichlet} in $D$  by
	\begin{equation}
		u(x)%
		=\pi^{-2} \sin(\pi \alpha/2)%
		\int_{D^{\texttt{c}}}%
		\left( \frac{1-|x|^2} {|y|^2-1} \right)^{\alpha/2} \frac{1} {|y-x|^2} \exp(-|x-y|^2) \,{\rm d}y,
		\qquad x\in D.
		\label{eq:ana_T4.4}
	\end{equation}
	This integral can be computed numerically via a quadrature approximation. Here, instead of a fixed number of samples, the number of samples is taken adaptively based on a tolerance $\varepsilon$ for the computed sample standard deviation.
	Figure \ref{fig:test5} shows the results with $y=(2,0)$ and tolerance $\varepsilon=10^{-4}$ for evaluation of $u(0.6,0.6)$ as previously. The estimator standard deviation and absolute error exhibit no obvious trend, whereas the sample variance peaks at about $\alpha=0.6$. Also at this value, the largest number of samples is needed to satisfy the tolerance. Despite the sample variance decreasing after $\alpha=0.6$, there is an increasing trend in the amount of work required. This implies that the increase in the number of steps with $\alpha$ (see Figure \ref{fig:steps}) dominates and therefore a solution point of accuracy $10^{-4}$ is computationally more costly for larger values of $\alpha$.
	\begin{figure}
		\centering
		\includegraphics{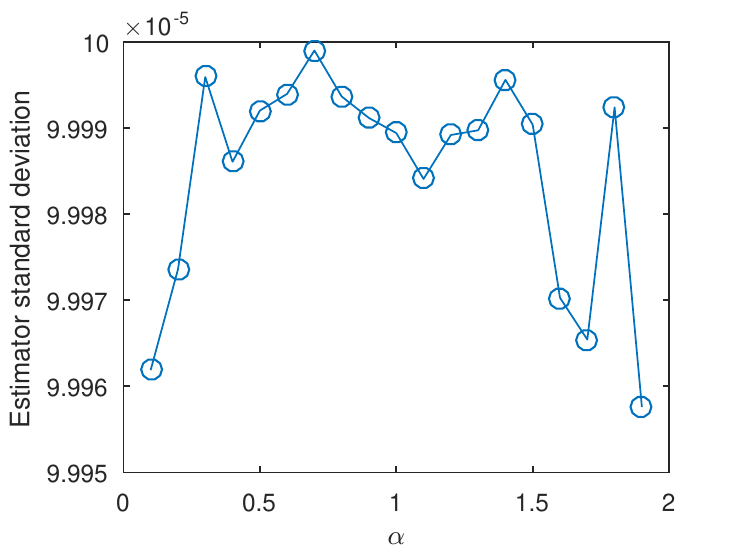}\quad
		\includegraphics{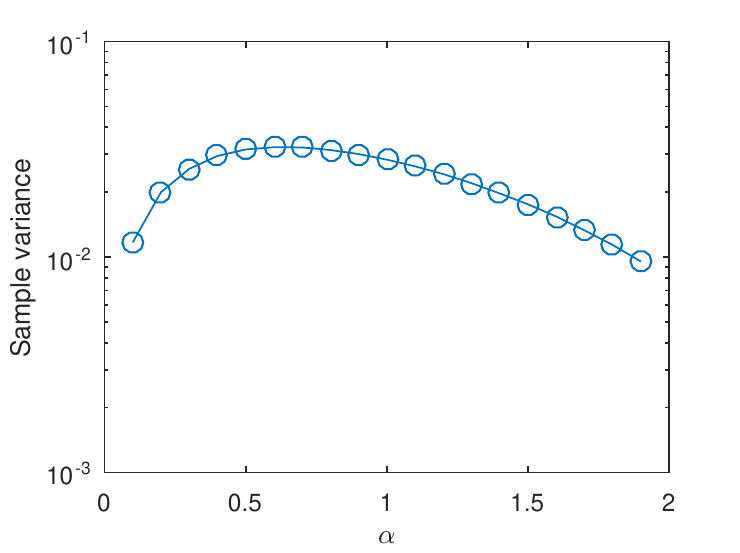}\\
		\includegraphics{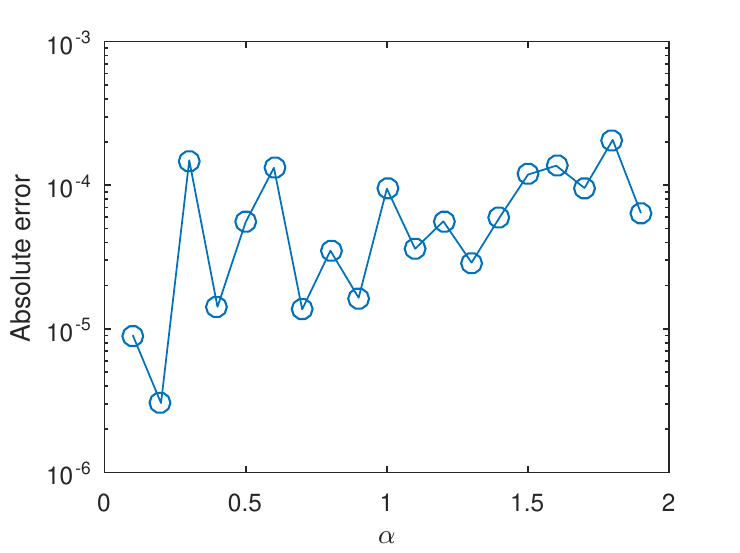}\quad
		\includegraphics{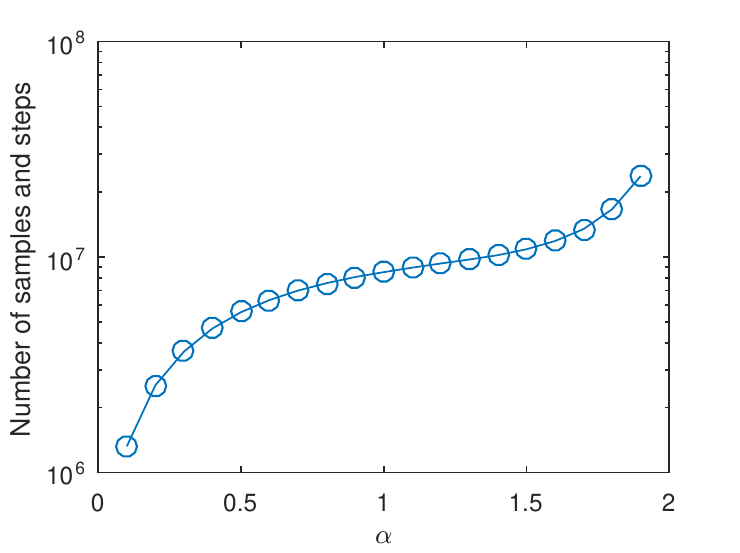}
																																																																																																																																																										                          
		\caption{Example simulation with the walk-on-spheres algorithm for \eqref{aDirichlet} based on desired tolerance of $10^{-4}$. From top left to bottom right, we see the standard deviation of the estimator, the sample variance, the absolute error (using a quadrature approximation for \eqref{eq:ana_T4.4} for the reference value), and the amount work (number of samples $\times$ mean number of steps).} \label{fig:test5}
	\end{figure}
																																																																														            
	\subsection{Non-constant source term} Suppose that, again in the context of \eqref{aDirichlet_g}, we again take  $D$  to be equal to the unit ball and the source term equal to 
	\[
		%
		\rhs(x)=2^{\alpha} \Gamma(2+\alpha/2) \Gamma(1+\alpha/2) (1-(1+\alpha/2) \|{x}\|^2), \qquad x\in D,
	\]
	and zero exterior data $\bc=0$.
	This has the exact solution $u(x)=\max\{0, 1-\|x\|^2\}^{1+\alpha/2}$; cf. \autocite{Dyda2012-kl}. The behaviour of the algorithm is shown in Figure~\ref{fig:dyda}. As expected, we again observe no obvious trend in estimator standard deviation and absolute error. The sample variance of sums of Monte Carlo-generated integrals increases with $\alpha$ as does the number of samples accordingly. Work required grows with $\alpha$ as in Figure \ref{fig:test5}, but with a slightly steeper trend. Notice that accuracy of $10^{-4}$ for the inhomogeneous part of the solution would demand a lot more work than the homogeneous part in Figure \ref{fig:test5}.
	\begin{figure}
		\centering
		\includegraphics{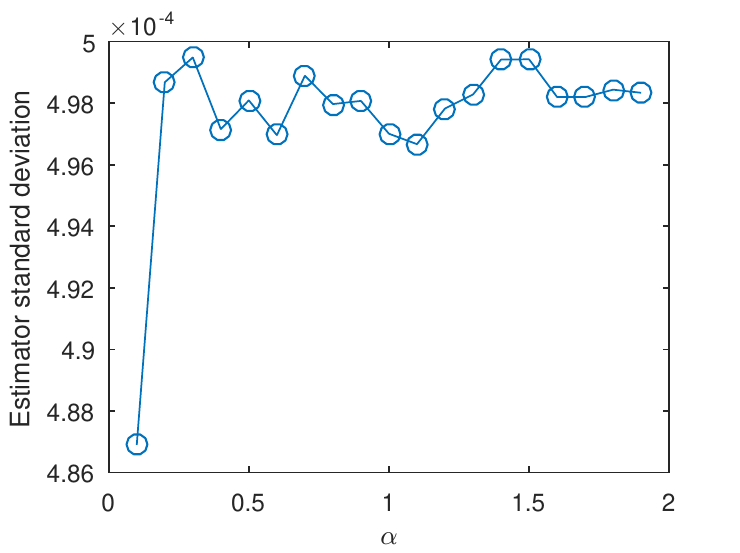}\quad
		\includegraphics{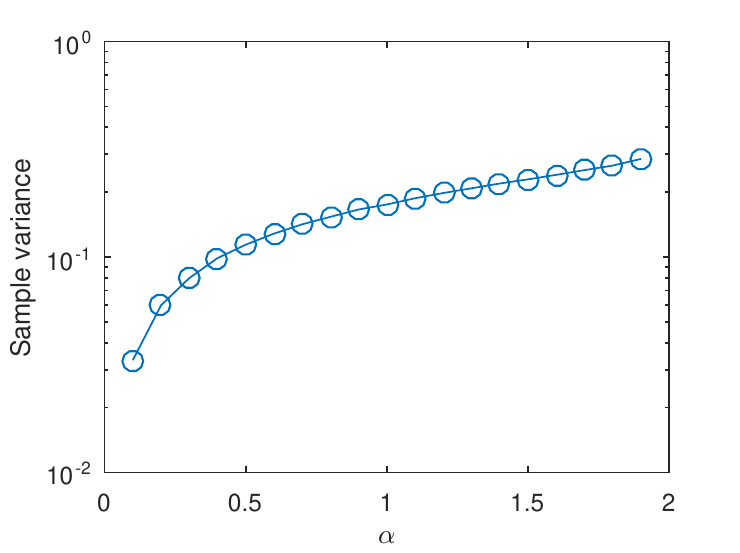}\\
		\includegraphics{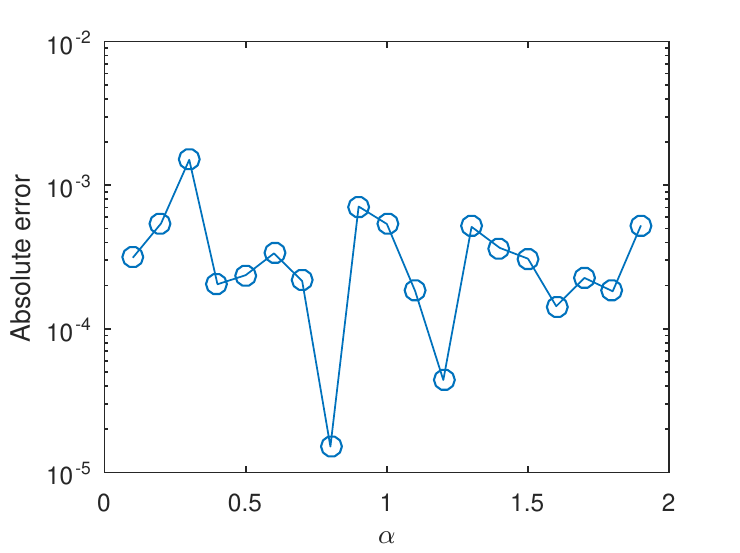}\quad
		\includegraphics{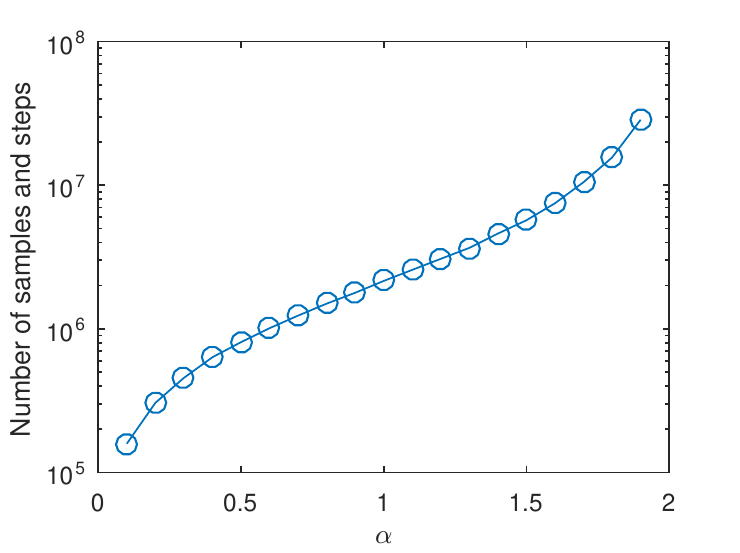}
																																																																																																																																																										                          
		\caption{Example simulation with the walk-on-spheres algorithm for \eqref{aDirichlet_g} based on desired tolerance of $10^{-3}$ and $n_{R,\Theta}=1000$. From top left to bottom right, we see the standard deviation of the estimator, the sample variance, the absolute error, and the amount work (number of samples $\times$ mean number of steps).} \label{fig:dyda}
	\end{figure}
																																																																														            
	\begin{figure}
		\centering
		\includegraphics{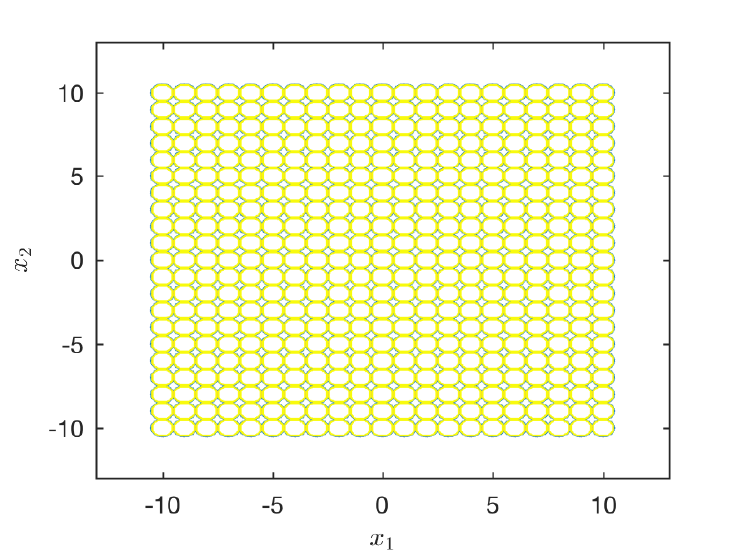}

		\caption{`Swiss cheese' domain (interior of balls).} \label{fig:cheese}
	\end{figure}
																																																																														            
	\subsection{Distribution of the number of steps in convex and non-convex domains}
	In previous sections, a large focus was put on deriving upper bounds and limiting distributions for $N$. Here we provide numerical support for these theoretical results. The walk-on-spheres algorithm was simulated inside a unit-ball domain centred at the origin as well as inside a domain of a hundred touching unit balls centred at points $(i,j)$, $i,j=-10,\dots,10$, the so-called `Swiss cheese' domain as shown in Figure \ref{fig:cheese}.  The first represents a convex domain whereas the latter a non-convex one. The algorithm was started at a point $x=(\sqrt{0.29},-\sqrt{0.7})$ which lies very close to the boundary in both domains. This point was chosen as numerical simulations in a unit-ball domain revealed that the mean number of steps decreases with increasing distance from the boundary of the starting point. Theorem \ref{main} states that $N$ is stochastically dominated by a geometric distribution with parameter $p(\alpha,d)$. In two dimensions, we are able to numerically compute $p(\alpha,2)$ since it is the solution to \eqref{aDirichlet} with $D = B(0, 1)$, $\bc(x)=\mathbf{1}_{\{x_1<-1\}}(x)$ and zero source term $\rhs=0$ as deduced from Corollary \ref{indicators}. We computed values of $p(\alpha,2)$ for different $\alpha$ to accuracy $10^{-4}$. 
																																																																														            
	The left-hand histogram in Figure \ref{fig:hist} confirms stochastic dominance of $\Gamma$ and an exponentially decaying tail as stated in Remark 2 of Theorem \ref{main}. However, the right-hand histogram shows that this statement fails in the particular example of the Swiss cheese domain. 
	Moreover, the plot of the mean number of steps against $\alpha$ in Figure \ref{fig:steps} shows the observed value of $\mathbb{E}_x[N]$ is bounded above by $1/p(\alpha,2)$ for the unit-ball domain. On the other hand, this is not the case for the Swiss cheese domain,  where the observed value of $\mathbb{E}_x[N]$ exceeds $1/p(\alpha,2)$ for $\alpha$ in the range $(0.3,1.6)$. 
																																																																														            
	An explanation for why this is  happening might be as follows. At larger values of $\alpha$, the path of $X$ starts resembling that of a Brownian motion (albeit with a countable infinity of arbitrarily small discontinuities). The process $X$ is started inside a ball in the Swiss cheese. When it exits this ball, its exit position is relatively close to the boundary with high probability. Therefore the exit point of the ball containing the point of issue is more likely to be in the `cheese' (which would cause an end to the algorithm) and less likely to be inside another vacuous ball. Accordingly,  $\mathbb{E}_x[N]$ does not deviate largely from the example of a single ball. However, for small values of $\alpha$, exit points from the sphere containing the point of issue  have a higher probability to be far from the boundary, landing inside another vacuous ball, thereby requiring the algorithm to continue. In that case, the comparison with the case of exiting a single sphere breaks down. 
	\begin{figure}
		\centering
		\includegraphics{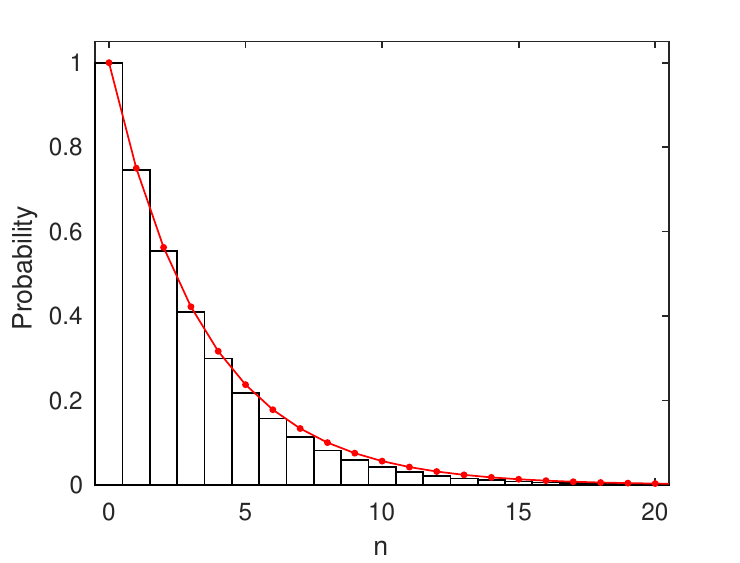}\quad
		\includegraphics{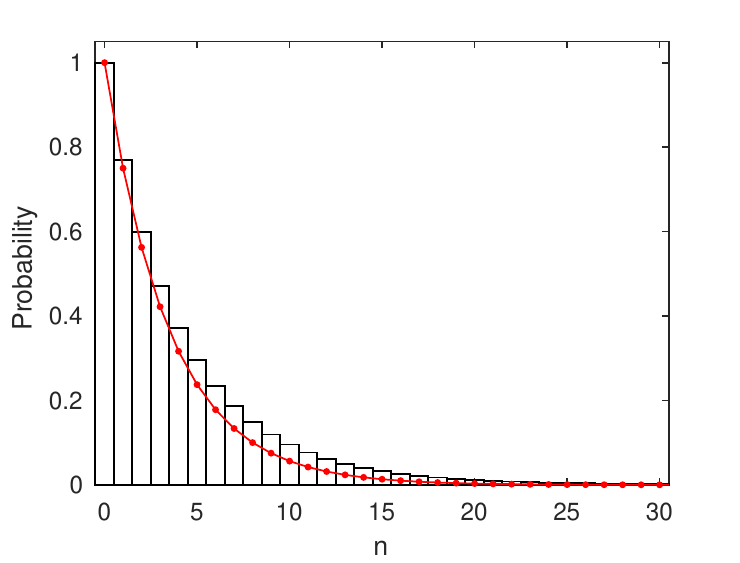}

		\caption{Histogram of the proportion of runs of the walk-on-spheres algorithm with $\alpha=1$ for which $N>n$ for the unit-ball domain (left) and the Swiss cheese domain (right). The red curve shows the tail of Geom($p(1,2)$), this is $(1-p(1,2))^n$, as in Remark 2 of Theorem~\ref{main}.} \label{fig:hist}
	\end{figure}
																																																																														            
	\begin{figure}
		\centering
		\includegraphics{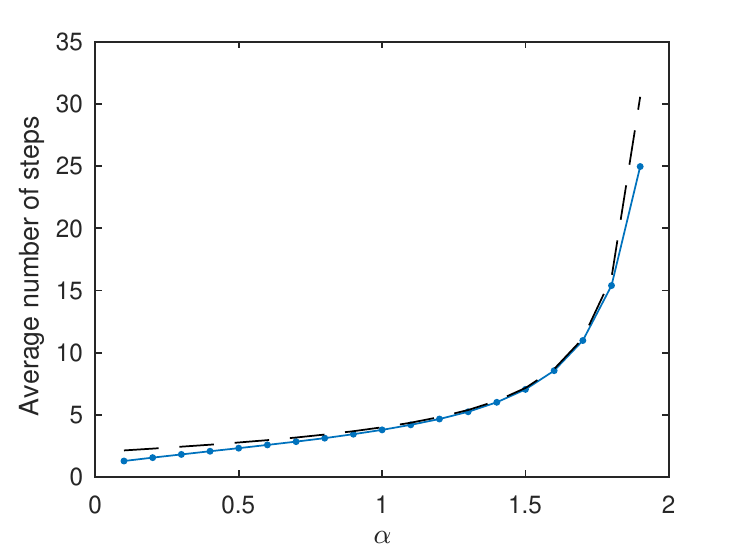}\quad
		\includegraphics{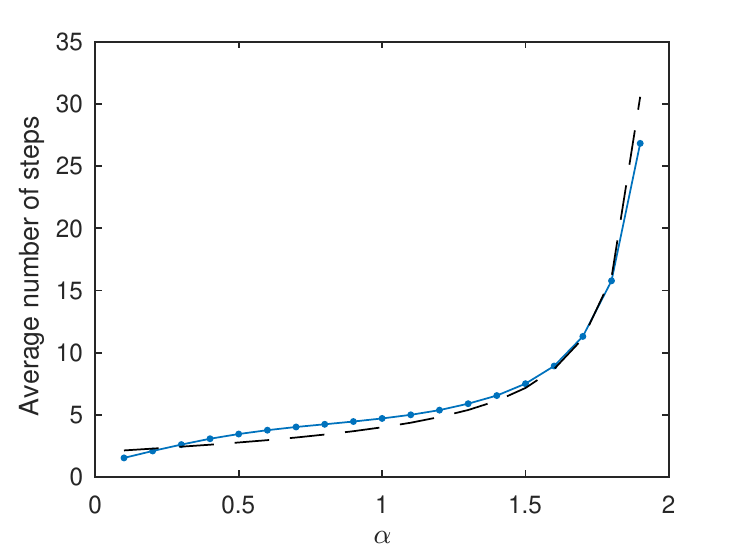}

		\caption{Mean number of steps for the walk-on-spheres algorithm started at $x=(\sqrt{0.29},-\sqrt{0.7})$ inside the circle domain (left) and inside the Swiss cheese domain (right). The dashed curve on both plots is $1/p(\alpha,2)$ as in Corollary \ref{indicators}.} \label{fig:steps}
	\end{figure}
																																																																														            
	\newpage
																																																																														            
	\section*{Appendix: Proof of Theorem \ref{hasacorr}}

	Our proof of Theorem \ref{hasacorr} uses heavily the joint conclusion of Theorems 2.10 and~3.2 in \autocite{bucur}, namely that the Theorem \ref{hasacorr} is true in the case that $D$ is a ball.  Our proof is otherwise constructive proving  existence and uniqueness separately. 
																																																																														            
	\bigskip
																																																																														            
	\noindent{\it Existence:}  On account of the fact that $D$ is bounded, we can define a ball of sufficiently large radius $R>0$, say $B^* = B(X_0, R)$, centred at $X_0$, such that $D$ is a subset of $B^*$ and hence  $\sigma_D \leq \sigma_{B^*}$ almost surely, irrespective of the initial position of $X$, where $\sigma_{B^*} = \inf\{t>0\colon X_t\not\in B^*\}$. In particular, thanks to stationary and independent increments, this upper bound for $\sigma_D$ does not depend on $X_0$ in law and $\sup_{x\in D}\mathbb{E}_x[\sigma_D]\leq \mathbb{E}_0[\sigma_{B^*}]<\infty$.

	Define for convenience $\upsilon(x)=\bc(x)$ for $x\in D^\mathrm{c}$
and	\begin{equation}
		\upsilon(x) %
		= \mathbb{E}_x[\bc(X_{\sigma_D})]%
		+ \mathbb{E}_x\left[\int_0^{\sigma_D} \rhs(X_s)\,{\rm d}s\right], \qquad x\in D,
		\label{v}
	\end{equation}
	where $\bc$ and $\rhs$ satisfy the assumptions of the theorem. We want to prove that $\upsilon$ is bounded and continuous on $\overline{D}$. For the boundedness of $\upsilon$, we prove the boundedness of the two expectations in its definition. 
																																																																													            
	First note that, for all $x\in D$,
	\begin{align}
		\mathbb{E}_x\bp{|\bc(X_{\sigma_D})|} 
		  & = \mathbb{E}_x\bp{\abs{\bc(X_{\sigma_D})}\mathbf{1}_{(\sigma_D = \sigma_{B^*})}} +  \mathbb{E}_x\bp{\abs{\bc(X_{\sigma_D})}\mathbf{1}_{(\sigma_D <\sigma_{B^*})}} \notag                     \\
		  & \leq \mathbb{E}_x\bp{\abs{\bc(X_{\sigma_{B^*}})}} + \sup_{x\in B^*\backslash D}|\bc(x)|\notag                                                                                            \\
		  & =\mathbb{E}_0\bp{|\bc(x + B^*X_{\sigma_{B(0,1)}})|} + \sup_{x\in B^*\backslash D}|\bc(x)|\notag                                                                                      \\
		  & =\pi^{-(d/2+1)}\,\Gamma(d/2)\,\sin(\pi\alpha/2)\, \int_{|y|>1} \frac{|\bc(x + B^* y)|}{\left|1-|y|^2\right|^{\alpha/2}|y|^{d}} \,{\rm d}y+ \sup_{x\in B^*\backslash D}|\bc(x)|\notag \\
		  & =C \int_{\mathbb{R}^d} \frac{|\bc(z)|}{1+|z|^{d+\alpha}} \,{\rm d}y+ \sup_{x\in B^*\backslash D}|\bc(x)|<\infty,                                                                     
		\label{worksforsquared}
	\end{align}
	for some constant $C\in(0,\infty)$ that does not depend on $x$ (this is ensured thanks to the boundedness of $D$). 
	In the  inequality, we have used the fact that, on $\{\sigma_D <\sigma_{B^*}\}$, we have $X_{\sigma_D}\in B^*\backslash D$, moreover, that, as a continuous function on $\mathbb{R}^d$, $\bc$ is bounded in $B^*\backslash D$.
	In the second equality, we have used spatial homogeneity and the scaling property of stable processes. In the third equality, we have used Theorem \ref{BGR}.
	The fourth equality follows by changing variables to $z = x + B^*y$ in the integral, appropriately estimating the denominator and  the assumption that $\bc$  is continuous and in $L^1_\alpha(D^\mathrm{c})$.
																																																																													            
	The boundedness of  $\rhs$ on $D$ and the uniform finite mean of $\sigma_D$ ensures that the second expectation in the definition of $\upsilon$ is bounded on $\overline{D}$.
	We claim that $\upsilon$ is  continuous in $\mathbb{R}^d$ and belongs to $L^1_\alpha(\mathbb{R}^d)$.
	Continuity of $\upsilon$ follows thanks to path regularity of $X$, the continuity of $\bc$, the openness of $D$ and the fact that $\omega\mapsto X_{\sigma_D}(\omega)$  and $\omega\mapsto\int_0^{\sigma_D(\omega)} \rhs(X_s(\omega))\,{\rm d}s$ are continuous in the Skorohod topology (for which it is important that $\omega\mapsto\sigma_D(\omega)$ is finite). Continuity is also a consequence of the classical potential analytic point of view, seeing the identity for $\upsilon$ in \eqref{v} in terms of Riesz potentials; see for example the classical texts of \cite{BH} or \cite{L}

	To check that  $\upsilon\in L^1_\alpha(\mathbb{R}^d)$, we need some estimates. For $x\in D^{\rm c}$, $\upsilon(x) = \bc(x)$ and hence, as $\bc\in L^1_\alpha(D^\mathrm{c})$, it suffices  to check that
	$
	\int_{D}|\upsilon(x)|/(1+|x|^{\alpha + d})\,{\rm d}x<\infty.$  However, this is trivial on account of the boundedness and continuity of $\upsilon$ on $\overline D$.


	Now fix $x'\in D$ and let $B(x')$ be the largest ball centred at $x'$ that is contained in $D$. 
	A simple application of the strong Markov property tells us that 
	\begin{align}
		\upsilon(x) & = \mathbb{E}_x\left[\mathbb{E}\left[\left.\bc(X_{\sigma_D}) + \int_0^{\sigma_{B(x')}} \rhs(X_s)\,{\rm d}s                
		+ \int_{\sigma_{B(x')}}^{\sigma_D} \rhs(X_s)\,{\rm d}s\right|\mathcal{F}_{\sigma_{B(x')}}\right]\right] \notag\\
		            & = \mathbb{E}_x\left[\upsilon (X_{\sigma_{B(x')}}) + \int_0^{\sigma_{B(x')}} \rhs(X_s)\,{\rm d}s\right], \qquad x\in D, 
		\label{localisedv}
	\end{align}
	where $(\mathcal{F}_t, t\geq 0)$ is the natural filtration generated by $X$.
	Thanks to the fact that Theorem \ref{hasacorr} is valid on balls, we see immediately that the right-hand side of \eqref{localisedv} is the unique  solution to 
	\begin{gather}
		\begin{aligned}
			-(-\Delta u)^{\alpha/2}u(x) & =-\rhs(x),        & \qquad  x & \in B(x'),         \\
			u(x)                        & = \upsilon(x), & x         & \in B(x')^{\rm c}. 
		\end{aligned} 
		\label{upsilondirichlet}
	\end{gather}
	That is to say, $\upsilon$ solves \eqref{upsilondirichlet}. Note that it is at this point in the argument that we are using the condition $\rhs\in C^{\alpha +\varepsilon}(\overline{D}) $.
	Since the solution to \eqref{upsilondirichlet} is  defined on $B(x')$ and $x'$ is chosen arbitrarily in $D$, we conclude that $\upsilon$ solves 
	\begin{gather}
		\begin{aligned}
			-(-\Delta u)^{\alpha/2}u(x) & =-\rhs(x),        & \qquad  x & \in D,         \\
			u(x)                        & = \upsilon(x), & x         & \in D^{\rm c}. 
		\end{aligned} 
		\label{upsilondirichletonD}
	\end{gather}
	On account of the fact that $\mathbb{P}_x(\sigma_D =0) = 1$ for all $x\in D^{\rm c}$, it follows that   $\upsilon = \bc$ on $D^{\rm c}$ and hence \eqref{upsilondirichletonD} is identical to \eqref{aDirichlet_g}.
																																																																													            
	\bigskip
																																																																													            
	\noindent{\it Uniqueness:}  Suppose that $\hat{u}$ solves \eqref{aDirichlet_g}, then, in particular, for any $x'\in D$, it must solve 
	\begin{gather*}
		\begin{aligned}
			-(-\Delta u)^{\alpha/2}u(x) & =-\rhs(x),       & \qquad  x & \in B(x'),         \\
			u(x)                        & = \hat{u}(x), & x         & \in B(x')^{\rm c}. 
		\end{aligned} 
	\end{gather*}
	As we know the Feynman--Kac representation of the solution to the above  fractional Poisson problem, thanks to Theorem 3.2 in \autocite{bucur} for domains which are balls, we are forced to conclude that 
	\begin{equation}
		\hat{u}(x) = \mathbb{E}_x\left[\hat{u} (X_{\sigma_{B(x')}}) + \int_0^{\sigma_{B(x')}} \rhs(X_s)\,{\rm d}s\right], \qquad x\in B(x'),\quad x'\in D.
		\label{fixedpointonspheres}
	\end{equation}
	Here again, we are implicitly using that $\rhs\in C^{\alpha +\varepsilon}(\overline{D})$ in the application of Theorem~3.2 of \autocite{bucur}.
	Let us now appeal to the same notation we have used for the walk-on-spheres. Specifically, recall the sequential exit times from maximally sized balls $\sigma_{B_k}$ for the walk-on-spheres which were defined in Section \ref{WoSfL}. We claim that 
	\[
		M_k\eqqcolon \hat{u} (X_{\sigma_{B_k}\wedge \sigma_D}) + \int_0^{\sigma_{B_k}\wedge \sigma_D }\rhs(X_s)\,{\rm d}s, \qquad k \geq 0,
	\]
	is a martingale. To see why, note that,  by the strong Markov property and then by  \eqref{fixedpointonspheres}, 
	\begin{align*}
		\mathbb{E}\left[M_{k+1}|\mathcal{G}_{k}\right] = & \, \mathbf{1}_{\{k<N\}}\left\{\left.\mathbb{E}_{x}\left[                                                                                  
		\hat{u} (X_{\sigma_{B(x)}}) + \int_0^{\sigma_{B(x)}}\rhs(X_s)\,{\rm d}s
		\right]\right|_{x = \smash{X_{\sigma_{B_k}}}} +  \int_0^{\sigma_{B_k}}\rhs(X_s)\,{\rm d}s\right\}\\
		                                                 & + \mathbf{1}_{\{k\geq N\}}\left\{                                                                                                         
		\hat{u} (X_{\sigma_D}) + \int_0^{ \sigma_D }\rhs(X_s)\,{\rm d}s
		\right\}\\
		=                                                & \, \mathbf{1}_{\{k<N\}}\left\{\hat{u}(X_{\sigma_{B_k}}) +  \int_0^{\sigma_{B_k}}\rhs(X_s)\,{\rm d}s\right\}+ \mathbf{1}_{\{k\geq N\}}\left\{ 
		\hat{u} (X_{\sigma_D}) + \int_0^{ \sigma_D }\rhs(X_s)\,{\rm d}s
		\right\}\\
		=                                                & \, \hat{u} (X_{\sigma_{B_k}\wedge \sigma_D}) + \int_0^{\sigma_{B_k}\wedge \sigma_D }\rhs(X_s)\,{\rm d}s                                      \\
		=                                                & \, M_k, \qquad k\geq 1,                                                                                                                   
	\end{align*}
	where $\mathcal{G}_k = \mathcal{F}_{\sigma_{B_k} \wedge\sigma_D}$, $k\geq 1$. For consistency, we may define $M_0 = \mathbb{E}_x[M_k] = \hat{u}(x)$ thanks to \eqref{fixedpointonspheres}.
																																																																													            
	Next, we appeal to the definition of $B^*$ and, in particular, that $\sigma_D\leq \sigma_{B^*}$, as well as the continuity of $\hat{u}$ to deduce that, for all $k\geq 0$, 
	\begin{align*}
		  & \left|\hat{u} (X_{\sigma_{B_k}\wedge \sigma_D}) + \int_0^{\sigma_{B_k}\wedge \sigma_D }\rhs(X_s)\,{\rm d}s\right| \\
		  & \leq \left|\hat{u} (X_{\sigma_{B^*}})\right|\,\mathbf{1}_{\{\sigma_{B_k}\wedge \sigma_D =  \sigma_{B^*}\}}     
		+\sup_{y\in B^*}\left|\hat{u} (y)\right|\,\mathbf{1}_{\{\sigma_{B_k}\wedge \sigma_D < \sigma_{B^*}\}} 
		+\sup_{y\in D} \left|\rhs(y)\right|\,\sigma_{D}\\
		  & \leq  \left|\bc(X_{\sigma_{B^*}})\right| + c_1 + c_2\sigma_{B^*},                                                
	\end{align*}
	where $c_1,c_2$ are constants.
	We know that for each fixed $x\in D$, $\mathbb{E}_x[\sigma_{B^*}]<\infty$ and, moreover, from Theorem \ref{BGR}, after scaling (see for example \eqref{rdy}), $\mathbb{E}_x[ |\bc(X_{\sigma_{B^*}})| ]<\infty$ as $\bc\in L^1_\alpha(D^\mathrm{c})$. Dominated convergence  allows us to deduce that $(M_k, k\geq 0)$ is a uniformly integrable martingale such that, for each fixed $x\in D$, 
	\begin{align*}
		\hat{u}(x) & = \lim_{k\to\infty}\mathbb{E}_x[M_k]                                                    \\
		           & = \mathbb{E}_x[\lim_{k\to\infty}M_k]                                                    \\
		           & = \mathbb{E}_x\left[\hat{u} (X_{\sigma_D}) + \int_0^{ \sigma_D }\rhs(X_s)\,{\rm d}s\right] \\
		           & = \mathbb{E}_x\left[\bc(X_{\sigma_D}) + \int_0^{ \sigma_D }\rhs(X_s)\,{\rm d}s\right],      
	\end{align*}
	where in the final equality we have used that $\hat{u} = \bc$ on $D^{\rm c}$. Uniqueness now follows. \hfill $\square$
																																																																													            
	\section*{Acknowledgements}     
	We would like to thank Mateusz Kwa\'sniki for pointing out a number of references to us and Alexander Freudenberg for a close reading of an earlier version of this manuscript.
	\def\bibfont{\normalfont\small}
	\printbibliography
																																																																													            
\end{document}